\documentclass[english,11pt]{article}

\usepackage[a4paper,bottom=3.5cm,top=2cm,left=2.54cm,right=2.54cm]{geometry}
\usepackage[english]{babel}
\usepackage[utf8]{inputenc}
\usepackage{url}
\usepackage[usenames]{color}
\usepackage{amsmath}
\usepackage{amssymb}
\usepackage{amsthm}
\usepackage{amsfonts}
\usepackage{bbm}
\usepackage{soul}
\usepackage{graphicx}
\usepackage{subcaption}
\usepackage{yfonts}
\usepackage[backend=bibtex,maxnames=4,url=false,style=numeric,isbn=false, giveninits=true,eprint=false,sorting=none,doi=false,date=year]{biblatex}
\usepackage{alphalph}
\usepackage{enumitem}
\usepackage{courier}
\usepackage[dont-mess-around]{fnpct}
\usepackage{etoolbox}
\usepackage{comment}
\usepackage{enumerate}
\usepackage{hhline}

\newcommand{\defeq}{:=}
\newcommand{\norm}[1]{\| #1\|}

\newcommand{\abs}[1]{\left|#1\right|}

\newcommand{\eps}{\varepsilon}
\newcommand{\one}{\mathbbm 1}
\newcommand{\charf}{\chi}
\newcommand{\bigO}{O}

\renewcommand{\leq}{\leqslant}
\renewcommand{\geq}{\geqslant}
\renewcommand{\phi}{\varphi}

\renewcommand{\b}{\bar}

\let\sp\relax
\newcommand{\sp}[1]{\langle #1 \rangle}
\newcommand{\lsc}{lower semicontinuous}

\newcommand{\adj}[1]{#1^*}

\newcommand{\subdiff}{\partial}

\newcommand{\reg}{\mathcal J}
\newcommand{\fid}{\mathcal H}
\newcommand{\dJ}{\subdiff \reg}
\newcommand{\RI}{\overline \R}
\newcommand{\Jminsol}{u^\dagger_\reg}

\newcommand{\musc}{\mu^\dagger}
\newcommand{\Dsymm}[1]{D_{#1}^\mathrm{symm}}
\newcommand{\D}[2]{D_{#1}^{#2}}
\newcommand{\infconv}{\,\square\,}
\renewcommand{\d}{\,\mathrm{d}}

\DeclareMathOperator{\dom}{dom}

\DeclareMathOperator{\interior}{int}

\DeclareMathOperator{\M}{\mathcal M}

\DeclareMathOperator{\Lip}{Lip}
\newcommand{\R}{\mathbb R}
\newcommand{\N}{\mathbb N}
\DeclareMathOperator{\KR}{KR}

\newcommand{\weakto}{\mathrel{\rightharpoonup}}

\newcommand{\wsto}{\weakto^*}

\newtheorem{theorem}{Theorem}
\numberwithin{theorem}{section}
\newtheorem{corollary}[theorem]{Corollary}
\newtheorem{lemma}[theorem]{Lemma}
\newtheorem{proposition}[theorem]{Proposition}

\theoremstyle{definition}
\newtheorem{definition}[theorem]{Definition}
\newtheorem{assumption}{Assumption}
\newtheorem*{assumption*}{Assumption}
\newtheorem{remark}[theorem]{Remark}
\newtheorem*{remark*}{Remark}
\newtheorem*{definition*}{Definition}
\newtheorem{example}[theorem]{Example}

\newcommand{\U}{\mathcal U}
\newcommand{\V}{\mathcal V}
\newcommand{\X}{\mathcal X}
\newcommand{\Y}{\mathcal Y}

\let\P\relax
\newcommand{\P}{\mathcal P}

\newcommand{\C}{{\mathcal C}}

\graphicspath{{pic/}}
\numberwithin{equation}{section}

\usepackage{tikz, pgfplots}
\pgfplotsset{compat=newest}
\pgfplotsset{plot coordinates/math parser=false}
\usetikzlibrary{external}
\usepackage[textsize=footnotesize,colorinlistoftodos]{todonotes}

\pdfoutput=1

\newcommand{\revtwo}[1]{#1}
\newcommand{\revthree}[1]{#1}

\bibliography{abbrevs.bib,BL.bib,IP.bib}

\title{Variational regularisation for inverse problems with imperfect forward operators and general noise models}

\author{Leon Bungert\footnote{Department Mathematik, University of Erlangen-N\"urnberg, Cauerstrasse~11, 91058 Erlangen, Germany, \mbox{leon.bungert;martin.burger@fau.de}} \and Martin Burger\footnotemark[1] \and Yury Korolev\footnote{Department of Applied Mathematics and Theoretical Physics, University of Cambridge, Wilberforce Road, Cambridge CB3 0WA, UK, \mbox{yk362;cbs31@cam.ac.uk}} \and Carola-Bibiane Sch\"onlieb\footnotemark[2]}
\date{\today}

\begin{document}

\maketitle

\begin{abstract}
We study variational regularisation methods for inverse problems with imperfect forward operators whose errors can be modelled by order intervals in a partial order of a Banach lattice.
We carry out analysis with respect to existence and convex duality for general data fidelity terms and regularisation functionals. 
Both for a-priori and a-posteriori parameter choice rules, we obtain convergence rates of the regularised solutions in terms of Bregman distances.
Our results apply to fidelity terms such as Wasserstein distances, $\phi$-divergences, norms, as well as sums and infimal convolutions of those.
\end{abstract}


\tableofcontents

\section{Introduction}



We consider linear inverse problems
 \begin{equation}\label{eq:Au=f}
    Au = \b f,
\end{equation}
where $A \colon \X \to \Y$ is a linear bounded operator (referred to as the forward operator or the forward model) acting between two Banach spaces $\X$ and $\Y$. The exact measurement $\b f$ is typically not available and only a noisy version of it $f_\delta$ is known along with an estimate of the noise level $\delta$. Since the inversion of~\eqref{eq:Au=f} is often unstable with respect to noise and hence ill-posed, it requires regularisation. Variational regularisation replaces solving~\eqref{eq:Au=f} by the following optimisation problem
\begin{equation}\label{eq:var_reg_1}
\min_{u \in \X} \frac1\alpha \fid(Au \mid f_\delta) + \reg(u),    \end{equation}
where $\fid(\cdot \mid f)$ is a so-called \emph{data fidelity} function that models statistical properties of the noise in $f$
and $\reg(\cdot)$ is a \emph{regularisation} functional  that stabilises the inversion. The \emph{regularisation  parameter} $\alpha>0$ balances the influence of the data fidelity and the regularisation. The amount of noise $\delta$ in the measurement $f_\delta$ is assumed to be such that
\begin{equation}\label{ineq:noise_level}
    \fid(\b f \mid f_\delta) \leq \delta.
\end{equation}

The fidelity function often depends only on the difference  of the arguments, i.e. $\fid( v \mid f) = h(v -f)$ for some function $h$. The most common example is $\fid(v \mid f) = \frac12\norm{v - f}^2$. There are, however, cases when the fidelity function depends on its arguments in a more complicated manner; an example is the Kullback--Leibler divergence that is used to model Poisson noise~\cite{Le_Poisson_noise:2007}, where $\fid(v \mid f) = \int (v \log\frac{v}{f} - (v-f))\d x$ \revtwo{(see also the review paper~\cite{hohage2016poisson})}. 
Problems with general fidelity functions were analysed in~\cite{Poschl:2008,Benning_Burger_general_fid_2011}.

To guarantee convergence of the minimisers of~\eqref{eq:var_reg_1} to a solution of~\eqref{eq:Au=f} as the noise level $\delta$ decreases, the regularisation parameter $\alpha$ needs to be chosen as a function of the measurement noise $\alpha=\alpha(\delta)$ (a-priori parameter choices) or of the measurement itself and of measurement noise $\alpha=\alpha(f_\delta,\delta)$ (a-posteriori parameter choices). For a-priori parameter choice rules, convergence rates for solutions of~\eqref{eq:var_reg_1} in different scenarios have been obtained, e.g., in~\cite{Engl:1989,Burger_Osher:2004,Resmerita:2005,Hofmann:2007,Grasmair:2011}. 
A classical a-posteriori parameter choice rule is the so-called discrepancy principle, originally introduced in~\cite{Morozov:1966} and later studied in, e.g.,~\cite{Bonesky:2008,Anzengruber_Ramlau:2009,Sixou:2018}. Roughly speaking, it consists in choosing $\alpha=\alpha(f_\delta,\delta)$ such that the following equation is satisfied
\begin{equation*}
    \fid(Au^\alpha \mid f_\delta) = \delta,
\end{equation*}
where $u^\alpha$ is the solution of~\eqref{eq:var_reg_1} corresponding to the regularisation parameter $\alpha$.

In many applications, not only the measurement $f_\delta$ is noisy, but also the forward operator $A$ that generated the data is not precisely known. Errors in the operator may come from the uncertainty in some model-related parameters such as the point-spread function of a microscope, simplified model geometry and/or discretisation. 
A classical approach to modelling errors in the forward operator assumes an error estimate in the operator norm, i.e.
\begin{equation}\label{eq:A_h}
    \norm{A_h-A}_{\mathcal L(\X,\Y)} \leq h,
\end{equation}
where $A_h \colon \X \to \Y$ is a linear bounded operator that we have numerical access to and $h \geq 0$ describes the approximation error (e.g.,~\cite{NeuSch90,TGSYag,Poeschl:2010,Bleyer_Ramlau:2013}). To guarantee convergence in this setting, the parameter $\alpha$ needs to be chosen as the function of $\delta$ and $h$ (a-priori choice rules) or of $\delta$, $h$, $f_\delta$ and $A_h$ (a-posteriori choice rules). Generalisations of the discrepancy principle to this setting are available~\cite{GONCHARSKII:1973,hofmann:1986,lu:2010}, but they usually rely on a triangle inequality that $\fid(\cdot \mid f)$ needs to satisfy.

An alternative approach to modelling operator errors using order intervals in Banach lattices was proposed in~\cite{Kor_Yag_IP:2013,Kor_IP:2014,mb_yk_jr:2019}. It assumes that the spaces $\X$ and $\Y$ have a lattice structure~\cite{Meyer-Nieberg} and, instead of~\eqref{eq:A_h}, lower and upper bounds for the operator are available
\begin{equation}\label{eq:Al_Au}
    A^l \leq A \leq A^u,
\end{equation}
where the inequalities are understood in the sense of a partial order for linear operators, i.e.
\begin{equation}\label{eq:po_oper}
    A^l u \leq Au \leq A^u u \quad \text{for all $u \geq 0$}.
\end{equation}
The inequalities in~\eqref{eq:po_oper} are understood in an abstract sense of a Banach lattice; which for $L^p$ spaces means inequality almost everywhere. 
In order for the partial order bounds~\eqref{eq:Al_Au} to be well-defined, we assume that $A \colon \X \to \Y$ is a regular operator~\cite{Meyer-Nieberg}, i.e. that it can be written as a difference of two positive operators, $A = A_1 - A_2$, where for any $u \geq 0$ it holds that $A_{1,2} u \geq 0$. 
Some examples of regular operators will be given later.

The approach~\eqref{eq:po_oper} to describing errors in the forward operator was studied in the context of the residual method in the case $\Y=L^\infty$ when the data fidelity is a characteristic function of a norm ball
\begin{equation}\label{eq:resid_meth_fid}
    \fid(\cdot \mid f_\delta) = \charf_{\norm{\cdot - f}_\infty \leq \delta}.
\end{equation}
 In this case, one solves the following problem
\begin{equation}\label{eq:var_po}
    \min_{u} \reg(u) \quad \text{s.t. $A^l u \leq f^u, \quad A^u u \geq f^l$,}
\end{equation}
where $f^l \defeq f_\delta - \delta \one$ and $f^u \defeq f_\delta + \delta \one$ are pointwise (a.e.) lower and upper bounds for the exact data $\b f$ in~\eqref{eq:Au=f} such that $f^l \leq \b f \leq f^u$ and $\one$ is the constant one-function. For comparison, with the data term~\eqref{eq:resid_meth_fid} and without an operator error,~\eqref{eq:var_reg_1} translates into
\begin{equation}\label{eq:resid_meth_1}
    \min_{u} \reg(u) \quad \text{s.t. $f^l \leq A u \leq f^u$,}
\end{equation}
where the constraint is equivalent to $\norm{Au - f_\delta}_\infty \leq \delta$. (In~\cite{AG_YK_TV:2016}, a connection is made between the lower and upper bounds $f^l,f^u$ and confidence intervals.)



One can show that the partial order based condition~\eqref{eq:Al_Au} implies the norm based condition~\eqref{eq:A_h}. Indeed, given $A^l,A^u$ as in~\eqref{eq:Al_Au}, one defines
\begin{equation*}
A_h \defeq \frac{A^u+A^l}{2}, \quad h \defeq \frac{\norm{A^u-A^l}}{2}.    
\end{equation*}
It can be readily verified that the so defined $A_h$ satisfies~\eqref{eq:A_h}. The opposite implication is, in general, wrong. Hence, if an estimate~\eqref{eq:Al_Au} is available, it allows one to describe the operator error more precisely and one may expect better reconstructions. Indeed, it was found in~\cite{mb_yk_jr:2019} that solving~\eqref{eq:var_reg_1} with $\fid(Au \mid f_\delta) = \norm{Au-f_\delta}_\infty$ and $\alpha$ chosen according to a generalised discrepancy principle~\cite{GONCHARSKII:1973} based on~\eqref{eq:A_h} produces overregularised solutions compared to~\eqref{eq:var_po}, i.e. the generalised discrepancy principle tends to overestimate the regularisation parameter. One of the reasons for this is the use of the triangle inequality to account for~\eqref{eq:A_h}, which makes the estimates not sharp, in general.



\revthree{The motivation for this paper is two-fold. First, we want to extend the approach~\eqref{eq:Al_Au},~\eqref{eq:var_po} to a broader class of fidelity terms than the characteristic function of a ball and more general data spaces than $L^\infty$. 
We also aim at a unified analysis of problems with fidelities that don't satisfy a triangle-type inequality, which is interesting in its own right. Our proofs mostly rely on convex analysis and duality. }



\paragraph{Setup.}
We consider the inverse problem~\eqref{eq:Au=f}, where $\X = \U^*$ and $\Y = \V^*$ are 
duals of Banach lattices $\U$ and $\V$, respectively. We assume that the partial order on $\Y$ is induced by the partial order in $\V$ as follows: $y \geq 0 \iff \sp{y,v} \geq 0$  $\forall v \in \V, \,\, v \geq 0$ (cf.  Lemma~\ref{lem:dual_order} in the appendix).

Furthermore, we assume that~\eqref{eq:Au=f} 
possesses a non-negative $\reg$-minimising solution $\Jminsol$, i.e.
\begin{equation}\label{eq:J_minimizing}
    A\Jminsol = \b f, \quad \Jminsol \geq 0 \quad \text{and} \quad \reg(\Jminsol) \leq \reg(u) \quad \text{for all $u$ such that $Au = \b f$.}
\end{equation}

We propose the following extension of~\eqref{eq:var_reg_1} to the case when the forward operator is known only up to the order interval given in~\eqref{eq:Al_Au}
\begin{equation}\label{eq:opt_prb}
    \min_{\substack{u \in \X \\ v \in \Y}} \frac{1}{\alpha} \fid(v \mid f_\delta) +  \reg(u) \quad \text{s.t. $A^l u \leq v \leq A^u u$},
\end{equation}
where $\reg \colon \X \to \RI_+$ and $\fid(\cdot \mid f) \colon \Y \to \RI_+$ (as a function of its first argument) are assumed proper, convex and weakly-* \lsc{} (cf. Assumption~\ref{ass:reg_fid}).


\paragraph{Main contribution.}
In this work we study convergence of solutions of~\eqref{eq:opt_prb} to a $\reg$-minimising solution of~\eqref{eq:Au=f} as the noise in data and operators decreases, and obtain convergence rates in one-sided Bregman distances with respect to $\reg$.
We also give conditions when~\eqref{eq:opt_prb} admits strong duality, in which case the convergence rates translate to symmetric Bregman distances.
Furthermore, we analyse an a-posteriori parameter choice rule based on a discrepancy principle for~\eqref{eq:opt_prb}.

Our results apply inter alia to general $\phi$-divergences, as for instance the Kullback-Leibler divergence, and coercive fidelities such as powers of norms or Wasserstein distances from optimal transport.
In addition, we also obtain rates for sums and infimal convolutions of different fidelities, as used for instance in mixed-noise removal. 
\revthree{Even for exact operators, our analysis goes beyond the state of the art in problems with fidelity terms that lack a triangle-type inequality.
}


\paragraph{Structure of the paper.}
In Section \ref{sec:primal_and_dual} we study existence of solutions of the problem~\eqref{eq:opt_prb} and its dual and establish sufficient conditions for strong duality. In Section~\ref{sec:conv_an} we derive convergence rates for a-priori parameter choice rules. In Section~\ref{sec:discr_pr} we formulate a discrepancy principle for the problem~\eqref{eq:opt_prb} and also obtain convergence rates. 
For readers' convenience, we present some background material on Banach lattices in the appendix.

\subsubsection*{Examples of regular operators} 
Below, we give some examples of regular operators and discuss how lower and upper bounds in the sense of~\eqref{eq:Al_Au}--\eqref{eq:po_oper} can be obtained.
\begin{example}
If $\Y$ is an abstract maximum space (a generalisation of $L^\infty$) or if $\X$ is an abstract Lebesgue space (a generalisation of $L^1$) then all linear bounded operators are regular, i.e. they can be written as a difference of two positive operators. More details can be found in the appendix.
\end{example}
\begin{example}[Integral operators -- perturbations of the kernel] \label{ex:int-op-kernel} 
Let $A \colon L^p(\Omega) \to L^q(\Omega)$ ($\Omega \subset \R^d$ bounded, $p,q \geq 1$) be an integral operator with a $(p,q)$-bounded kernel $k$~\cite{schachermayer1981lp},
\begin{equation}\label{eq:int_op}
    Au(x) \defeq \int_\Omega k(x,\xi) u(\xi) \, \d\xi.
\end{equation}
The operator $A$ can be written as
\begin{equation*}
    A = A_+ - A_-, \quad A_\pm \defeq \int_\Omega k(x,\xi)_\pm u(\xi) \, \d\xi,
\end{equation*}
where $k_+$ and $k_-$ are the positive and the negative parts of $k$ (in the a.e. sense in $\Omega \times \Omega$). Clearly, $A_\pm$ are positive and $A$ is regular. \\
\indent Suppose that the kernel is corrupted by an unknown $(p,q)$-bounded perturbation such that we only know pointwise lower and upper bounds for $k$,
\begin{equation}\label{eq:kl_ku}
    k^l(x,\xi) \leq k(x,\xi) \leq k^u(x,\xi) \quad \text{a.e. in $\Omega \times \Omega$}.
\end{equation}
Then lower and upper operators in the sense of~\eqref{eq:Al_Au} are given by
\begin{equation*}
    A^l u(x) \defeq \int_\Omega k^l(x,\xi) u(\xi) \, \d\xi, \quad A^u u(x) \defeq \int_\Omega k^u(x,\xi) u(\xi) \, \d\xi.
\end{equation*}
It should be noted that the bounds~\eqref{eq:kl_ku} are of a deterministic nature. They could arise, for example, if the kernel depends on additional parameters $\theta \in \Theta$, i.e. $k(x,\xi) = k_\theta(x,\xi)$. If reconstructing the unknown parameter $\theta$ is not of independent interest, the dependence on it can be eliminated by defining
\begin{equation*}
    k^l(x,\xi) \defeq \inf_{\theta \in \Theta} k_\theta(x,\xi), \quad k^u(x,\xi) \defeq \sup_{\theta \in \Theta} k_\theta(x,\xi),
\end{equation*}
provided the suprema and infima are finite for a.e. $x,\xi$ and $k^{l,u}$ are $(p,q)$-bounded.
\end{example}
\begin{example}[Integral operators -- discretisation]
Let the operator $A$ be as defined in Example~\ref{ex:int-op-kernel} on an interval $\Omega\subset\R$ and consider its approximation by Riemann sums. In particular, let $S^l_n(x)$ and $S^u_n(x)$ denote the lower and upper Riemann sums in~\eqref{eq:int_op} obtained using an $n$-point discretisation. Then these sums define lower and upper operators in the sense of~\eqref{eq:Al_Au},
\begin{equation*}
    A^l_n u(x) \defeq S^l_n(x), \quad  A^u_n u(x) \defeq S^u_n(x).
\end{equation*}
As we refine the discretisation (i.e. $n \to \infty$), these bounds converge pointwise to $Au(x)$.
\end{example}
\begin{example}[Integration with respect to a vector-valued measure]\label{ex:vector-valued-measures}
Example~\ref{ex:int-op-kernel} can be generalised as follows. Let $\mu \in \M(\Omega;Y)$ be a vector-valued Radon measure~\cite{diestel_uhl:1977}, where $\Omega$ is a compact metric space and $Y$ is a Banach lattice with the Radon-Nikod\'ym property. Define partial order on $\M(\Omega;Y)$ as follows
\begin{equation}\label{eq:po-vector-measures}
    \mu \geq_{\M} 0 \iff \mu(E) \geq_Y 0 \quad \text{for all $\mu$-measurable subsets $E \subset \Omega$}.
\end{equation}
Let $A \colon \C(\Omega) \to Y$ be defined as follows
\begin{equation*}
    Au \defeq \int_\Omega u \, \d\mu.
\end{equation*}
Since $Y$ is a lattice, it is clear that $A$ is regular. Lower and upper bounds $\mu^l \leq_{\M} \mu \leq_{\M} \mu^u$ in the sense of~\eqref{eq:po-vector-measures} define lower and upper operators $A^{l,u}$ in the sense of~\eqref{eq:Al_Au}.
\end{example}
\begin{example}[1D source identification]
We consider the operator $A:\M([0,1]) \to \C([0,1])$, $A:u\mapsto\phi$, where $\phi$ solves
\begin{equation*}
    \begin{cases}
    -(a\phi')' &= u, \quad $on $(0,1),\\
    \phi(0) &= 0,\\
    \phi'(0) &=0.
    \end{cases}
\end{equation*}
Here $a:[0,1]\to\R$ is a continuous function which meets $a\geq a_0>0$ on $[0,1]$ and $u\in\M([0,1])$ is a Radon measure with integrable antiderivative $U(x):=\int_0^x \d u$. Integrating the equation yields
\begin{equation*}
        Au(x) = \phi(x) = -\int_0^x \frac{U(y)}{a(y)} \d y.
\end{equation*}
Clearly, $A \leq 0$ and hence regular. Hence, if $\underline{a},\overline{a}:[0,1]\to\R$ are continuous functions such that $\underline{a} \leq a \leq \overline{a}$ on $[0,1]$ and $\underline{a}\geq \underline{a}_0>0$ on $[0,1]$, we can define operators
\begin{eqnarray*}
    A^l u(x) &=& -\int_0^x \frac{U(y)}{\underline{a}(y)} \d y, \\
    A^u u(x) &=& -\int_0^x \frac{U(y)}{\overline{a}(y)} \d y,
\end{eqnarray*}
which meet $A^l u \leq Au \leq A^u u$ for $u\geq 0$ (and hence $U \geq 0$). If $\norm{\overline{a}-\underline{a}}_\C \to 0$, then $A^{l,u}$ converge to $A$ in the operator norm. \\
\indent If one defines the operator $A$ on $L^1((0,1))$ instead of $\M([0,1])$, the antiderivative $U$ is continuous and one can approximate the integrals in $A^l$ and $A^u$ with lower and upper Riemann sums, respectively. 
This gives rise to operators $A^l_n$ and $A^u_n$ such that $A^l_n\leq A^l \leq A \leq A^u \leq A^u_n$. If then additionally $n\to \infty$, the operators $A^{l,u}_n$ converge to $A$. \\
\indent Note that a similar approach can be used for estimating the diffusivity $a$ for a given source term. 
In this case, however, the forward operator $A$ becomes non-linear. This would require an extension of our theory.
\end{example}
\begin{example}[Conditional expectations]
Let $\Omega$ be a separable metric space and $(\Omega,\Sigma,\mu)$ be a probability space.  
Let $B \subset \Sigma$ be a sub-$\sigma$-algebra of $\Sigma$ and let $\{E_i\}_{i=1}^\infty$ be its minimal generator (which exists, since $\Omega$ is separable). The \emph{conditional expectation operator} $A \colon L^p_\mu(\Omega) \to L^p_\mu(\Omega)$ is defined as follows
\begin{equation*}
    Au \defeq \sum_{i=1}^\infty \frac{\int_{E_i} u \, \d\mu}{\mu(E_i)} \chi_{E_i},
\end{equation*}
under the convention $0/0=0$. Clearly, $A \geq 0$ and hence regular. \\
\indent If we allow $\mu$ to be a finite signed measure, then we can generalise the definition as follows
\begin{equation*}
    Au \defeq \sum_{i=1}^\infty \frac{\int_{E_i} u \, \d\mu}{\abs{\mu}(E_i)} \chi_{E_i},
\end{equation*}
where $\abs{\mu}$ is the total variation of $\mu$. Clearly,
\begin{equation*}
    A = A_+ - A_-, \quad A_\pm u \defeq \sum_{i=1}^\infty \frac{\int_{E_i} u \, \d\mu_\pm}{\abs{\mu}(E_i)} \chi_{E_i}
\end{equation*}
and $A_\pm \geq 0$, hence $A$ is regular.
In contrast to Example~\ref{ex:vector-valued-measures}, partial order bounds on $\mu$ in the sense of~\eqref{eq:po-vector-measures} do not translate into lower and upper bounds~\eqref{eq:po_oper} for $A$ since $A$ is not an integral operator (in particular, it is not linear in~$\mu$).
\end{example}

\section{Primal and Dual Problems}
In this section we establish existence of solutions to~\eqref{eq:opt_prb} using the direct method, where standard assumptions on the forward operators, the regularisation, and fidelity function will guarantee coercivity and lower semicontinuity.
Subsequently, we derive the dual maximisation problem and prove existence and strong duality under the additional assumption that the data space $\Y$ is an abstract maximum space.

\label{sec:primal_and_dual}
\subsection{Existence of a primal solution}
We make the following standard assumptions on the regularisation functional $\reg$, the fidelity function $\fid$, and the operators $A^{l,u}$.
\begin{assumption}\label{ass:reg_fid}
The regularisation functional $\reg \colon \X \to \RI_+$ is
\begin{itemize}
    \item proper, convex and weakly-* \lsc{};
    \item its non-empty sublevel sets $\{u \in \X \colon \reg(u) \leq C\}$ are weakly-* sequentially compact.
\end{itemize}
The fidelity function $\fid(\cdot \mid \cdot) \colon \Y \times \Y \to \RI_+$ is
\begin{itemize}
    \item proper, convex in its first argument and weakly-* \lsc{} jointly in both arguments;
    \item $\fid(v \mid f) = 0$ if and only if $v=f$.
\end{itemize}
\end{assumption}

\begin{assumption}\label{ass:op}
The operators $A,A^{l,u}:\X\to\Y$ are weak-* to weak-* continuous.
\end{assumption}

A sufficient condition for Assumption~\ref{ass:op} to hold is given in Lemma~\ref{lem:weak-star_cont} in the appendix.

\begin{theorem}\label{thm:primal_ex}
Suppose that Assumptions~\ref{ass:reg_fid} and~\ref{ass:op} hold true. Then~\eqref{eq:opt_prb} has a solution.
\end{theorem}
\begin{proof}
Consider a minimising sequence $(u_k,v_k)$. Due to Assumption~\ref{ass:reg_fid} there exists a convergent subsequence $u_k$ (that we don't relabel) such that
\begin{equation*}
    u_k \wsto u_\infty.
\end{equation*}
Then Assumption~\ref{ass:op} yields
\begin{equation*}
    A^{l,u} u_k \wsto A^{l,u} u_\infty.
\end{equation*}
From~\eqref{eq:opt_prb} we get that for all $k$
\begin{equation*}
    0 \leq v_k - A^l u_k \leq (A^u-A^l) u_k,
\end{equation*}
hence
\begin{equation*}
    \norm{v_k - A^l u_k} \leq \norm{(A^u-A^l)u_k}
\end{equation*}
and
\begin{equation*}
    \norm{v_k} \leq \norm{A^l u_k} + \norm{(A^u-A^l)u_k} \leq C,
\end{equation*}
since weakly-* convergent sequences are bounded.

Since $\Y$ is a dual of a separable Banach space $\V$, by the sequential Banach-Alaoglu theorem the sequence $v^k$ contains a weakly-* convergent subsequence $v_k$ (that we don't relabel) such that
\begin{equation*}
    v_k \wsto v_\infty.
\end{equation*}
Since both $A^{l,u}u_k$ and $v_k$ converge weakly-* and order intervals in $\Y$ are weakly-* closed due to Lemma~\ref{lem:dual_order}, we obtain that
\begin{equation*}
    A^l u_\infty \leq v_\infty \leq A^u u_\infty.
\end{equation*}
Hence $(u_\infty,v_\infty)$ is feasible for \eqref{eq:opt_prb}.
Furthermore, since $\reg(\cdot)$ and $\fid(\cdot \mid f)$ are weakly-* \lsc{}, we get that 
\begin{equation*}
    \frac1\alpha\fid(v_\infty \mid f) + \reg(u_\infty) \leq \liminf_{k \to \infty} \frac1\alpha\fid(v_k \mid f) + \reg(u_k) = \inf_{\substack{u \in \X \\ v \in \Y \\ A^l u \leq v \leq A^u u}} \frac1\alpha\fid(v \mid f) +  \reg(u). 
\end{equation*}
Therefore, $(u_\infty, v_\infty)$ is a solution of~\eqref{eq:opt_prb}.
\end{proof}


\subsection{Dual problem}
To simplify our notation, we introduce an operator $B \colon \X \to \Y \times \Y$
\begin{equation}\label{eq:def_B}
    Bu = 
    \begin{pmatrix}
    A^l u \\
    -A^u u
    \end{pmatrix}
\end{equation}
and an operator $E \colon \Y \to \Y \times \Y$
\begin{equation}\label{eq:def_E}
    Ev = 
    \begin{pmatrix}
    v \\
    -v
    \end{pmatrix}.
\end{equation}
With this notation we can rewrite~\eqref{eq:opt_prb} as follows
\begin{equation}\label{eq:primal}
    \min_{\substack{u \in \X \\ v \in \Y \\ Bu \leq Ev}} \frac{1}{\alpha}\fid(v \mid f) + \reg(u).
\end{equation}


\begin{proposition}
The (Lagrangian) dual problem of~\eqref{eq:primal} is given by
\begin{equation}\label{eq:dual}
    \sup_{\substack{\mu \in \Y^* \times \Y^* \\ \mu \geq 0}} -\frac1\alpha\fid^*(\alpha \adj{E}\mu \mid f) - \reg^*(-\adj{B} \mu).
\end{equation}
\end{proposition}
\begin{proof}
The Lagrangian of~\eqref{eq:primal} is given by
\begin{equation*}\label{eq:lagrangian}
    \mathcal L(u,v,\mu) = \frac1\alpha\fid(v \mid f) + \reg(u) + \sp{\mu,Bu-Ev},
\end{equation*}
where $\mu \in \Y^* \times \Y^*$, $\mu \geq 0$. Minimising the Lagrangian in $u$ and $v$, we obtain
\begin{eqnarray*}
    \inf_{u,v} \mathcal L(u,v,\mu) &=& \inf_{u,v} \frac1\alpha\fid(v \mid f) + \reg(u) + \sp{\mu,Bu-Ev} \\
    &=&  \inf_u \left[\reg(u) - \sp{-\adj{B}\mu,u} \right] + \frac1\alpha\inf_v \left[ \fid(v \mid f) - \sp{\alpha\adj{E}\mu,v} \right] \\
    &=& -\reg^*(-\adj{B} \mu) - \frac1\alpha\fid^*(\alpha \adj{E}\mu \mid f).
\end{eqnarray*}
Taking a supremum over $\mu \geq 0$ gives~\eqref{eq:dual}.
\end{proof}

It is well known (e.g.,~\cite{Bonnans_Shapiro_book:2000}) that
\begin{equation*}
    \inf_{\substack{u \in \X \\ v \in \Y \\ Bu \leq Ev}} \frac{1}{\alpha}\fid(v \mid f) + \reg(u) 
    \geq \sup_{\substack{\mu \in \Y^* \times \Y^* \\ \mu \geq 0}} -\frac1\alpha\fid^*(\alpha \adj{E}\mu \mid f) - \reg^*(-\adj{B} \mu),
\end{equation*}
which is referred to as \emph{weak duality}.

\begin{remark}\label{rem: fid_difference}
If the fidelity function depends only on the difference of its arguments, i.e. $\fid(\cdot \mid f) =  h(\cdot - f)$, then
\begin{equation*}
    \fid^*(\alpha \adj{E}\mu \mid f) =  h^*(\alpha \adj{E}\mu) + (\alpha\adj{E}\mu,f)
\end{equation*}
and problem~\eqref{eq:dual} becomes
\begin{equation}\label{eq:dual_diff}
    \sup_{\substack{\mu \geq 0 
    }
    } -\frac1\alpha h^*(\alpha \adj{E}\mu) - (\alpha\adj{E}\mu,f) - \reg^*(-\adj{B} \mu).
\end{equation}
If $h(\cdot) = \frac{1}{2}\norm{\cdot}_\Y^2$, we have $h^*(\cdot) = \frac{1}{2}\norm{\cdot}_{\Y^*}^2$ and hence we obtain the standard form (e.g.,~\cite{Peyre:2017})
\begin{equation*}
    \sup_{\substack{\mu \geq 0
    }} -\frac{\alpha}{2} \norm{\adj{E}\mu}_{\Y^*}^2 - (\adj{E}\mu,f) -\reg^*(-\adj{B} \mu) = - \inf_{\substack{\mu \geq 0 
    }} \frac{\alpha}{2} \norm{\adj{E}\mu}_{\Y^*}^2 + (\adj{E}\mu,f) +\reg^*(-\adj{B} \mu).
\end{equation*}
\end{remark}

\subsection{Existence of a dual solution and strong duality}
The goal of this section is to study the relationship between the primal problem~\eqref{eq:primal} and its dual~\eqref{eq:dual}, establishing strong duality and existence of a dual solution, and obtaining complementarity conditions for Lagrange multipliers associated with constraints in~\eqref{eq:primal}.

We will need the following result from~\cite[Theorem 2.165]{Bonnans_Shapiro_book:2000}.
\begin{theorem}[\cite{Bonnans_Shapiro_book:2000}] \label{thm:bonnans1}
Consider the following optimisation problem
\begin{equation}\label{thm:primal_prob}
\inf_{x \in X} g(x) \quad \text{s.t. $L x \in K$}, \tag{$\mathcal P$}
\end{equation}
and its dual
\begin{equation}\label{thm:dual_prob}
\sup_{y^* \in Y^*} - \charf_K^*(y^*)- g^*(-\adj{L}y^*), \tag{$\mathcal D$}
\end{equation}
where $X$ and $Y$ are Banach spaces, $L \colon X \to Y$ is a linear bounded operator, $\adj{L}$ its adjoint, $K \subset Y$ a closed convex set, and $g \colon X \to \R$ a proper convex \lsc{} function with convex conjugate $g^* \colon X^* \to \R$. The characteristic function of $K$ is denoted by $\charf_K(\cdot)$  and its convex conjugate (i.e. the support function of $K$) by $\charf_K^*(\cdot)$. Suppose that the following regularity condition is satisfied
\begin{equation}\label{eq:Robinson_reg}
    0 \in \interior(L (\dom g) - K).
\end{equation}
Then there is no duality gap between problems~\eqref{thm:primal_prob} and~\eqref{thm:dual_prob}. If the optimal value of~\eqref{thm:primal_prob} is finite, then the dual problem~\eqref{thm:dual_prob} has at least one solution $\b y^* \in Y^*$.
\end{theorem}
The regularity condition~\eqref{eq:Robinson_reg} is due to Robinson~\cite{Robinson_lin} and plays an important role in the stability of optimisation problems under perturbations of the feasible set~\cite{Bonnans_Shapiro_book:2000}.

To ensure that~\eqref{eq:Robinson_reg} is satisfied in the primal problem~\eqref{eq:primal}, we will need to assume that the positive cone in $\Y$ has a non-empty interior. 
This naturally leads to the concept of abstract maximum spaces~\cite{Meyer-Nieberg} which are a generalisation of $L^\infty(\Omega)$.

\begin{definition}\label{def:AM-space}
A Banach lattice $\Y$ with norm $\norm{\cdot}$ is called an \emph{AM-space} (abstract maximum space) if
\begin{equation*}
    \norm{x \vee y}=\norm{x} \vee \norm{y},\quad \forall x,y \geq 0.
\end{equation*}
An element $\one \in \Y$ which meets
\begin{equation*}
    \one \geq 0,\quad \norm{\one} = 1,\quad \norm{x}\leq 1 \implies \abs{x} \leq \one,
\end{equation*}
is called \emph{unit} of $\Y$.
Here $x \vee y$ and $\abs{x}$ denote the usual supremum and absolute value of elements in a Banach lattice (cf.~appendix).
\end{definition}

\begin{theorem}\label{thm:robinson_primal}
Let  $\Y$ be an AM-space with unit $\one$ and suppose that there exist $u_0 \in \dom(\reg)$ and $v_0 \in \dom(\fid(\cdot \mid f))$ such that
\begin{equation*}
    A^l u_0 + \eps \one \leq v_0 \leq A^u u_0 - \eps \one,
\end{equation*}
where $\eps >0$ is a constant.
Then Robinson's condition~\eqref{eq:Robinson_reg} is satisfied in the primal problem.
\end{theorem}
\begin{proof}
In the notation of Theorem~\ref{thm:bonnans1}, we have $X = \X \times \Y$, $g(u,v) \defeq \frac{1}{\alpha} \fid(v \mid f) + \reg(u)$, $L \defeq (B, -E)$ and $K = \Y_- \times \Y_-$ (where $\Y_-$ denotes the negative cone in $\Y$). 

Take an arbitrary $y = (y^1,y^2) \in \Y \times \Y$ with $\norm{y} \leq \eps$.
Without loss of generality we can choose the norm on $\Y \times \Y$ to be $\norm{y}=\max(\norm{y^1},\norm{y^2})$.
Hence, the definition of the unit implies
\begin{equation*}
    -\eps \one \leq y^{1,2} \leq \eps \one.
\end{equation*}
To show Robinson regularity, we need to write $y$ as 
\begin{equation}\label{eq:robinson_primal_y_represent}
    y = Bu - Ev - z
\end{equation}
for some $u \in \dom(\reg)$, $v \in \dom(\fid(\cdot \mid f))$ and $z =(z^1,z^2) \in \Y \times \Y$, $z^{1,2} \leq 0$. Writing this in terms of $A^l$ and $A^u$, we get
\begin{equation*}
    y^1 = A^l u - v - z^1, \quad y^2 = v - A^u u - z^2.
\end{equation*}
Take $u=u_0$ and $v=v_0$. Then
\begin{eqnarray*}
    z^1 &=& A^l u_0 - v_0 - y^1 \leq -\eps \one - y^1 \leq 0, \\ 
    z^2 &=& v_0 - A^u u_0 - y^2 \leq -\eps \one - y^2 \leq 0,
\end{eqnarray*}
and we can take $z^{1,2}$ as above to represent $y$ as in~\eqref{eq:robinson_primal_y_represent}. Hence, the Robinson condition~\eqref{eq:Robinson_reg} is satisfied.
\end{proof}
\begin{corollary}
Since the optimal value of the primal problem~\eqref{eq:primal} is finite, using Theorem~\ref{thm:bonnans1} we conclude that there exists a solution $\mu$ of the dual problem~\eqref{eq:dual} and there is no duality gap, i.e.
\begin{equation}\label{eq:no_gap}
    \frac{1}{\alpha}\fid(v \mid f) + \reg(u) = -\frac1\alpha\fid^*(\alpha \adj{E}\mu \mid f) -\reg^*(-\adj{B} \mu),
\end{equation}
where $(u,v)$ is a primal optimal solution. Moreover, from~\cite[Thm. 3.6]{Bonnans_Shapiro_book:2000} we conclude that $\mu$ is a Lagrange multiplier for the constraint $Bu \leq Ev$ in~\eqref{eq:primal} and the following complementarity condition holds
\begin{equation}\label{eq:complementarity}
    \sp{\mu, Bu-Ev} = 0.
\end{equation}
\end{corollary}

\begin{theorem}\label{thm:subgrads}
Let $\mu$ be an optimal solution of~\eqref{eq:dual} and $(u,v)$ be an optimal solution of~\eqref{eq:primal}. Then under the assumptions of Theorem~\ref{thm:robinson_primal} we have the following relations
\begin{equation*}
    - \adj{B}\mu \in \dJ(u),\qquad \alpha \adj{E} \mu \in \subdiff \fid(v \mid f).
\end{equation*}
\end{theorem}
\begin{proof}
Using the Fenchel--Young inequality, strong duality~\eqref{eq:no_gap} and the feasibility of $(u,v)$,
we obtain
\begin{eqnarray*}
    \sp{-\adj{B}\mu, u} &\leq& \reg(u) + \adj{\reg}(-\adj{B}\mu) \\
    &=& -\frac{1}{\alpha}\adj{\fid}(\alpha\adj{E}\mu \mid f)-\frac{1}{\alpha}\fid(v\mid f) \\
    &=& -\frac{1}{\alpha}\left[\adj{\fid}(\alpha\adj{E}\mu \mid f) + \fid(v\mid f) \right] \\
    &\leq& -\frac{1}{\alpha} \sp{\alpha\adj{E}\mu, v} \\
    &=& -\sp{\mu, Ev} \\
    &\leq& -\sp{\mu, Bu} \\
    &=& \sp{-\adj{B}\mu, u}.
\end{eqnarray*}
Hence, equality holds everywhere and we get that $-\adj{B}\mu\in\partial\reg(u)$ and $\alpha\adj{E}\mu\in\partial\fid(v \mid f)$.

\end{proof}

\section{Convergence Analysis}\label{sec:conv_an}

Having investigated well-posedness of the primal and dual problems, we can now prove convergence rates of solutions as the noise in the data and the operator tends to zero.
To this end we consider sequences 
\begin{subequations}\label{eq:convergence_all}
\begin{eqnarray}
    A^l_n, A^u_n \colon && A^l_n \leq A \leq A^u_n \quad \forall n, \\
    && \norm{A^u_n - A^l_n} \leq \eta_n \to 0 \quad \text{as $n \to \infty$}, \label{eq:convergence_oper} \\
    f_n, \delta_n \colon && \fid(\b f \mid f_n) \leq \delta_n \quad \forall n, \label{eq: delta_n} \\
    && \delta_n \to 0 \quad \text{as $n \to \infty$}, \label{eq:convergence_data} \\
    \alpha_n \colon && \alpha_n \to 0 \quad \text{as $n \to \infty$},
\end{eqnarray}
\end{subequations}
and corresponding sequences $(u_n, v_n)$ and $\mu_n$ which solve problems~\eqref{eq:primal} and~\eqref{eq:dual}, respectively. We are interested in studying the behaviour of $(u_n, v_n)$ as $n \to \infty$ and would like to prove that $u_n$ converges to a $\reg$-minimizing solution~$\Jminsol$ (cf.~\eqref{eq:J_minimizing}) whereas $v_n$ approaches the exact data~$\b f$.

\begin{remark}
If the fidelity function depends on the difference of the arguments, i.e. $\fid(\b f \mid f_n) = h(\b f - f_n)$, then it does not matter if we choose $\fid(\b f \mid f_n)$ or $\fid(f_n \mid \b f)$ in~\eqref{eq: delta_n}. For asymmetric fidelities such as the Kullback--Leibler divergence it does. If we think of the Kullback--Leibler divergence $D_{KL}(p \mid q)$ as the amount of information lost by using $q$ instead of $p$ (cf.~\cite{Burnham_2002}), then it actually makes sense to choose $\fid(\b f \mid f_n)$ in~\eqref{eq: delta_n}, i.e. to measure the amount of information lost by using the noisy measurement $f_n$ instead of the exact one $\b f$.
\end{remark}

We start with results that do not require the existence of a dual solution and are valid under general assumptions (cf. Theorem~\ref{thm:primal_ex}).

\subsection{Convergence of  primal solutions}
We consider a sequence of primal problems~\eqref{eq:primal} 
\begin{equation*}
    \min_{\substack{u, v \\ B_nu \leq Ev}} \frac1{\alpha_n} \fid(v \mid f_n) + \reg(u),
\end{equation*}
where $B_n \colon \X \to \Y \times \Y$ is defined as follows
\begin{equation*}
    B_n \defeq 
    \begin{pmatrix}
    A^l_n  \\
    -A^u_n 
    \end{pmatrix}.
\end{equation*}


Under Assumptions~\ref{ass:reg_fid} and~\ref{ass:op}, we obtain the following standard result.
\begin{theorem}\label{thm:primal_conv}
Suppose that the regularisation functional $\reg$ and the fidelity function $\fid$ satisfy Assumption~\ref{ass:reg_fid} and the operators $A,A_n^{l,u}: \X \to \Y$ satisfy Assumption~\ref{ass:op}. Suppose also that the regularisation parameter $\alpha_n$ is chosen such that
\begin{equation*}
    \alpha_n \to 0 \quad \text{and} \quad \frac{\delta_n}{\alpha_n} \to 0 \quad \text{as $n \to \infty$}.
\end{equation*}
Then any solution $u_n$ of the primal problem~\eqref{eq:primal} converges weakly-* to a $\reg$-minimising solution of~\eqref{eq:Au=f}
\begin{equation*}
    u_n \wsto \Jminsol,
\end{equation*}
and $v_n$ converges weakly-* to the exact data in~\eqref{eq:Au=f}
\begin{equation*}
    v_n \wsto \b f = A\Jminsol.
\end{equation*}
\end{theorem}
\begin{proof}
Comparing the value of the objective function at the optimum $(u_n,v_n)$ and $(\Jminsol,\b f)$ (which is a feasible point for all $n$), we get
\begin{equation}\label{eq:compare}
    \frac1{\alpha_n}\fid(v_n \mid f_n) + \reg(u_n) \leq \frac1{\alpha_n}\fid(\b f \mid f_n) + \reg(\Jminsol)
\end{equation}
and
\begin{equation}\label{eq:est_Jn}
    \reg(u_n) \leq \reg(\Jminsol) + \frac{1}{\alpha_n} \fid(\b f \mid f_n) \leq \reg(\Jminsol) + \frac{\delta_n}{\alpha_n}.
\end{equation}
Since $\frac{\delta_n}{\alpha_n} \to 0$, the value on the right hand side is bounded uniformly in $n$. Hence, since sublevel sets of $\reg$ are weakly-* sequentially compact, $u_n$ contains a weakly-* convergent subsequence (that we don't relabel) that converges to some $u_\infty \in \X$
\begin{equation*}
    u_n \wsto u_\infty.
\end{equation*}
Since $A$ is weak-* to weak-* continuous by assumption and $\norm{A^{l,u}_n - A} \to 0$, we get that
\begin{equation*}
    A^l_n u_n \wsto A u_\infty \quad \text{and} \quad A^u_n u_n \wsto A u_\infty.
\end{equation*}
Since $(u_n,v_n)$ is feasible in~\eqref{eq:primal} for all $n$, it holds
\begin{equation*}
    A^l_n u_n \leq v_n \leq A^u_n u_n.
\end{equation*}
Using weak-* closedness of order intervals (cf. Lemma~\ref{lem:dual_order}), we infer
\begin{equation}\label{eq:conv_v_n}
    v_n \wsto A u_\infty.
\end{equation}
From~\eqref{eq:compare} we get that
\begin{equation*}
    \fid(v_n \mid f_n) \leq \fid(\b f \mid f_n) + \alpha_n \reg(\Jminsol) \leq \delta_n + \alpha_n \reg(\Jminsol)  \to 0.
\end{equation*}
Since $\fid(\cdot \mid \cdot)$ is \lsc{} jointly in both arguments, we obtain
\begin{equation*}
    \fid(Au_\infty \mid \b f) \leq \liminf_{n \to \infty} \fid(v_n \mid f_n)  = 0
\end{equation*}
and hence 
\begin{equation*}
    Au_\infty = \b f.
\end{equation*}
Therefore, by~\eqref{eq:conv_v_n} we have
\begin{equation*}
    v_n \wsto \b f.
\end{equation*}
Since $\reg$ is \lsc{},~\eqref{eq:est_Jn} implies that
\begin{equation*}
    \reg(u_\infty) \leq \liminf_{n \to \infty} \reg(u_n) \leq \reg(\Jminsol),
\end{equation*}
hence $u_\infty$ is a $\reg$-minimising solution.
\end{proof}

\subsection{Convergence rates}
In modern variational regularisation, (generalised) Bregman distances are typically used to study convergence of approximate solutions~\cite{Benning_Burger_modern:2018}.
\begin{definition} 
For a proper convex functional $\reg$ the generalised Bregman distance between $u,w \in \X$ corresponding to the subgradient $p \in \dJ(w)$ is defined as follows
	\begin{equation*}
	\D{\reg}{p}(u,w) \defeq \reg(u) - \reg(w) - \sp{p,u-w},
	\end{equation*}
	where $\dJ(w)$ denotes the subdifferential of $\reg$ at $w \in \X$. The symmetric Bregman distance between $u$ and $w$ corresponding to $q \in \dJ(u)$ and $p \in \dJ(w)$ is defined as follows
	\begin{equation*}
	\Dsymm{\reg}(u,w) \defeq \D{\reg}{p}(u,w) + D_\reg^q(w,u) = \sp{q-p,u-w}.
	\end{equation*}
\end{definition}
Bregman distances do not define a metric since they do not satisfy the triangle inequality and $\Dsymm{\reg}(u,w) = 0$ does not imply $u=w$.

To obtain convergence rates, we will need to make an additional assumption on the regularity of the $\reg$-minimising solution $\Jminsol$ called the \emph{source condition}. 
There are several variants of the source condition (e.g.,~\cite{engl:1996,Burger_Osher:2004,Hoffmann_Yamamoto_sorce_cond:2010}); we will use the variant from~\cite{Burger_Osher:2004}, which in our notation can be written as follows
\begin{assumption}[Source condition]\label{ass:sc}
There exists $\musc \in \Y^* \times \Y^*$, $\musc \geq 0$, such that
\begin{equation}\label{eq:source_condition}
    -\adj{B} \mu^\dagger \in \dJ(\Jminsol).
\end{equation}
\end{assumption}

\begin{remark}
The source condition~\eqref{eq:source_condition} is equivalent to the standard one
\begin{equation}\label{eq:source_condition2}
    A^*\omega \in \dJ(\Jminsol), \quad \omega \in \Y^*.
\end{equation}
Indeed, since $B = 
    \begin{pmatrix}
    A  \\
    -A 
    \end{pmatrix}$ and $\musc = (\musc_1, \musc_2)$ with $\musc_{1,2} \in \Y^*_+$, we get that
    \begin{equation*}
        -B^*\musc = A^*(\musc_2-\musc_1),
    \end{equation*}
    which implies~\eqref{eq:source_condition2} with $\omega \defeq \musc_2 - \musc_2$. For the converse implication we note that since $\Y^*$ is a lattice, we can write an arbitrary $\omega \in \Y^*$ as follows
    \begin{equation*}
        \omega = \omega_+ - \omega_-,
    \end{equation*}
    where $\omega_\pm \in \Y^*_+$. Hence,~\eqref{eq:source_condition2} implies~\eqref{eq:source_condition} with $\musc \defeq (\omega_-, \omega_+)$.
\end{remark}

\subsubsection{Convergence rates in a one-sided Bregman distance} 
We start with a convergence rate in a one-sided Bregman distance $D^{p^\dagger}_\reg$, where $p^\dagger \defeq -B^*\mu^\dagger$ is the subgradient from the source condition~\eqref{eq:source_condition}.
\begin{theorem}\label{thm:conv_rate_one_sided}
Let assumptions of Theorem~\ref{thm:primal_ex} and Assumption~\ref{ass:sc}  be satisfied and~\eqref{eq:convergence_all} hold. Then the following estimate holds
\begin{equation}\label{eq:Breg_dist_one-sided}
    D^{p^\dagger}_\reg (u_n, \Jminsol) \leq \frac{\delta_n}{\alpha_n}  + \frac{1}{\alpha_n}[\fid^*(\alpha_n E^* \mu^\dagger \mid f_n) - \sp{\alpha_n E^*\mu^\dagger, \b f}]  + C\eta_n,
\end{equation}
where $p^\dagger=-\adj{B}\musc$ is the subgradient from Assumption~\ref{ass:sc}.
\end{theorem}
\begin{proof}
 We start with the following estimate
\begin{eqnarray}
    D^{p^\dagger}_\reg (u_n, \Jminsol) &=& \reg(u_n) - \reg(\Jminsol) - \sp{-B^* \mu^\dagger, u_n-\Jminsol} \nonumber\\
    &=& \reg(u_n) - \reg(\Jminsol)  + \sp{\musc, Bu_n} - \sp{\musc, B\Jminsol} \nonumber \\
    &=& \reg(u_n) - \reg(\Jminsol) + \sp{\mu^\dagger, B_n u_n}  - \sp{\musc, B\Jminsol} + \sp{\mu^\dagger, (B-B_n) u_n}  \nonumber  \\
    &\leq&  \reg(u_n) - \reg(\Jminsol) + \sp{\mu^\dagger, B_n u_n} - \sp{\musc, B\Jminsol}  + C\eta_n \nonumber \\ 
    &\leq& \reg(u_n) - \reg(\Jminsol)  + \sp{\mu^\dagger, E v_n} - \sp{\musc, B\Jminsol}  + C\eta_n, \label{eq:Breg_dist_one-sided_intermediate}
\end{eqnarray}
where $\eta_n$ is as defined in~\eqref{eq:convergence_oper} and we used the fact that $B_n u_n \leq E v_n$. 
Since $(u_n,v_n)$ is primal optimal and $(\Jminsol, \b f)$ is feasible, we get that
\begin{equation*}
    \frac{1}{\alpha_n}\fid(v_n \mid f_n) + \reg(u_n) \leq \frac{1}{\alpha_n}\fid(\b f \mid f_n) + \reg(\Jminsol) \leq \frac{\delta_n}{\alpha_n} + \reg(\Jminsol)
\end{equation*}
and therefore
\begin{eqnarray*}
    D^{p^\dagger}_\reg (u_n, \Jminsol) 
    &\leq& \frac{\delta_n}{\alpha_n} - \frac{1}{\alpha_n}\fid(v_n \mid f_n) + \sp{\mu^\dagger, E v_n}  - \sp{\mu^\dagger,B\Jminsol}  + C\eta_n \\
    &=& \frac{\delta_n}{\alpha_n}  - \frac{1}{\alpha_n}\fid(v_n \mid f_n) + \sp{\mu^\dagger, E v_n} - \sp{\mu^\dagger,E \b f}  + C\eta_n \\
    &=& \frac{\delta_n}{\alpha_n} + \frac{1}{\alpha_n}\left[ \sp{\alpha_n E^*\mu^\dagger, v_n}-\fid(v_n \mid f_n)\right] - \sp{E^*\mu^\dagger, \b f} + C\eta_n.
\end{eqnarray*}
By the Fenchel-Young inequality, the term in the brackets is bounded by $\fid^*(\alpha_n E^* \mu^\dagger \mid f_n)$, hence
\begin{eqnarray*}
    D^{p^\dagger}_\reg (u_n, \Jminsol) &\leq& \frac{\delta_n}{\alpha_n} + \frac{1}{\alpha_n}\fid^*(\alpha_n E^* \mu^\dagger \mid f_n) - \sp{E^*\mu^\dagger, \b f}  + C\eta_n \nonumber \\
    &=& \frac{\delta_n}{\alpha_n}  + \frac{1}{\alpha_n}\left[\fid^*(\alpha_n E^* \mu^\dagger \mid f_n) - \sp{\alpha_n E^*\mu^\dagger, \b f}\right]  + C\eta_n. 
\end{eqnarray*}
\end{proof}


\subsubsection{Convergence rates in a symmetric Bregman distance} 
Under a stronger assumption that $\Y$ is an AM-space (cf. Theorem~\ref{thm:robinson_primal}), we can obtain an estimate in a symmetric Bregman distance.
\begin{theorem}\label{thm:conv_rate_symm}
Let assumptions of Theorem~\ref{thm:robinson_primal} and Assumption~\ref{ass:sc} be satisfied and~\eqref{eq:convergence_all} hold. Then the following estimate holds
\begin{equation}\label{eq:Breg_dist_symm}
    \Dsymm{\reg} (u_n, \Jminsol) \leq \frac{\delta_n}{\alpha_n}  + \frac{1}{\alpha_n}\left[\fid^*(\alpha_n E^* \mu^\dagger \mid f_n) - \sp{\alpha_n E^*\mu^\dagger, \b f}\right]  + C\eta_n,
\end{equation}
where the symmetric Bregman distance corresponds to the subgradients $p^\dagger \defeq -B^*\mu^\dagger \in \dJ(\Jminsol)$ from Assumption~\ref{ass:sc} and $p_n \defeq -B_n^*\mu_n \in \dJ(u_n)$.
\end{theorem}
\begin{proof}
The symmetric Bregman distance between $u_n$ and $\Jminsol$ is given by
\begin{eqnarray*}
    \Dsymm{\reg} (u_n,\Jminsol) &=& \sp{-\adj{B} \musc + B_n^* \mu_n,\Jminsol-u_n} \\
    &=& \sp{\mu_n, B_n\Jminsol - B_n u_n} - \sp{\musc, B\Jminsol - B u_n}.
\end{eqnarray*}
Since the pair $(\Jminsol,\b f)$ is feasible for all $n$, we get that $B_n\Jminsol \leq E \b f$. It is also evident that $B \Jminsol = E \b f$. Combining this with the complementarity condition~\eqref{eq:complementarity}, we obtain
\begin{eqnarray*}
    \Dsymm{\reg} (u_n,\Jminsol)    
    &=& \sp{\mu_n, E \b f - E v_n} + \sp{\musc,  B u_n - E \b f} \\
    &=& \sp{\mu_n, E \b f - E v_n} + \sp{\musc,  B_n u_n - E \b f} + \sp{\musc,  (B-B_n) u_n}.
\end{eqnarray*}
Since the pair $(u_n,v_n)$ is also feasible, we get that $B_nu_n \leq E v_n$ and hence
\begin{eqnarray*}
    \Dsymm{\reg} (u_n,\Jminsol)
    & \leq & \sp{\mu_n, E \b f - E v_n} + \sp{\musc, Ev_n - E \b f}  + \sp{\musc,  (B-B_n) u_n} \nonumber \\
    & \leq & \sp{\adj{E}\mu_n, \b f - v_n} - \sp{\adj{E}\musc, \b f - v_n} + \norm{\musc}\norm{u_n}\norm{B-B_n} \nonumber \\
    & \leq & \sp{\adj{E}\mu_n, \b f - v_n} - \sp{\adj{E}\musc, \b f - v_n} + C\eta_n,
    \label{Dsymm_est1}
\end{eqnarray*}
where $\norm{u_n}$ is bounded due to Theorem~\ref{thm:primal_conv}. 
From the Fenchel--Young inequality and Theorem~\ref{thm:subgrads} we get that
\begin{subequations}
\begin{eqnarray}
    \sp{\alpha_n \adj{E}\mu_n, \b f} & \leq & \fid(\b f \mid f_n) + \fid^*(\alpha_n \adj{E}\mu_n \mid f_n), \label{Fenchel_Young_est1} \\
    \sp{\alpha_n \adj{E}\mu_n, v_n} & = & \fid(v_n \mid f_n) + \fid^*(\alpha_n \adj{E}\mu_n \mid f_n), \label{Fenchel_Young_est2} \\
    \sp{\alpha_n \adj{E}\musc, v_n} & \leq & \fid(v_n \mid f_n) + \fid^*(\alpha_n \adj{E}\musc \mid f_n), \label{Fenchel_Young_est3}
\end{eqnarray}
\end{subequations}
hence
\begin{eqnarray*}
    \alpha_n \Dsymm{\reg} (u_n,\Jminsol) & \leq & \fid(\b f \mid f_n) - \fid(v_n \mid f_n) - \sp{\alpha_n \adj{E}\musc, \b f - v_n} + \alpha_n C\eta_n  \\ 
    & \leq & \delta_n - \fid(v_n \mid f_n) - \sp{\alpha_n \adj{E}\musc, \b f - v_n} + \alpha_n C\eta_n \nonumber \\
    & \leq & \delta_n + \fid^*(\alpha_n \adj{E}\musc \mid f_n) - \sp{\alpha_n \adj{E}\musc, \b f} + \alpha_n C\eta_n, \nonumber
\end{eqnarray*}
which yields the desired estimate upon dividing by $\alpha_n$.
\end{proof}

\subsection{Applications to different fidelity terms}\label{sec:applications_apriori}
To apply Theorems~\ref{thm:conv_rate_one_sided} or~\ref{thm:conv_rate_symm}, we need to study the term $\fid^*(\alpha_n E^* \mu^\dagger \mid f_n) - \sp{\alpha_n E^*\mu^\dagger, \b f}$ separately for each fidelity term. 
\subsubsection{$\phi$-divergences}

Let $\phi \colon (0,\infty) \to \R$ be a convex function. For two probability measures $\rho,\nu$ on $\Omega$ with $\rho \ll \nu$ the $\phi$-divergence (often called $f$-divergence) is defined as follows
\begin{equation}
    d_\phi(\rho \mid \nu) \defeq \int_\Omega \phi\left(\frac{\d\rho}{\d\nu}\right) \d\nu,
\end{equation}
where $\phi(1)=0$.
We refer to~\cite{liese2006divergences} for many examples and fundamental properties of $\phi$-divergences.
Since $\rho$ and $\nu$ have unit mass, function $\phi$ is only determined up to the additive term $c(x-1)$ for $c\in\R$.
In particular, since $\phi$ is convex and meets $\phi(1)=0$, it is straightforward to see that one can always find a suitable $c\in\R$ such that $\phi(x)+c(x-1) \geq 0$ for all $x>0$.
Hence, we will without loss of generality assume that $\phi \geq 0$.

We take $\Y = \M(\Omega)$ to be space of Radon measures on $\Omega$ equipped with the total variation norm and consider
\begin{equation}\label{eq:fid_phi-div}
    \fid(v \mid f) \defeq 
    \begin{cases}
    d_\phi(v \mid f) , \quad &\text{if }v\in\P(\Omega), \,\, v \ll f,\\
    \infty,\quad &\text{else,}
    \end{cases}
\end{equation}
where $\P(\Omega) \subset \M(\Omega)$ is the set of probability measures and $f \in \P(\Omega)$. 

\revthree{
We estimate the convex conjugate of $\fid(\rho \mid \nu)$ as follows:
\begin{eqnarray}
    \adj{\fid}(h \mid \nu) &=& \sup_{\rho \ll \nu} \sp{h, \rho} - \fid(\rho \mid \nu) \notag\\
    &=& \sup_{\rho \ll \nu} \int_\Omega  \left(h\frac{\d\rho}{\d\nu} - \phi\left(\frac{\d\rho}{\d\nu}\right) \right) \d\nu \notag\\
    &=& \sup_{f \in L^1_+(\Omega)} \int_\Omega \left(h(x) f(x) - \phi\left(f(x)\right) \right) \d\nu(x) \notag\\
    &\leq& \int_\Omega \sup_{y > 0} \left[h(x) y - \phi\left(y\right) \right] \d\nu(x) \notag\\
    &=& \int_\Omega \adj{\phi}(h) \d\nu \notag\\
    &=& \sp{\adj{\phi}(h), \nu}\label{eq:H_star_phi_div},
\end{eqnarray}
for any $h \in \C(\Omega)$.}

Since $\phi(1)=0$ and $\phi\geq 0$, we know that $\phi^*(0)=0$ and $\phi^*(x)\geq x$. Indeed, we have $\phi^*(0) = \sup_x -\phi(x) = - \inf_x \phi(x)=0$ and, by the Fenchel-Young inequality,  $\adj{\phi}(x)\geq x-\phi(1)=x$.
This motivates us to assume
\begin{equation}\label{eq:phi_star}
    \phi^*(x)=x+r(x),
\end{equation}
where $r(x)/x \to 0$ as $x\to 0$.
This is satisfied in many cases (examples will be provided later on). 


\begin{theorem}
Let $\fid(\cdot \mid \cdot)$ be as defined in~\eqref{eq:fid_phi-div} and let the assumptions of Theorem~\ref{thm:conv_rate_one_sided} be satisfied. 
Suppose that $E^*\musc\in\C(\Omega)$, where $\musc$ is the source element from Assumption~\ref{ass:sc}, and that~\eqref{eq:phi_star} holds. Then the following convergence rate holds
\begin{equation}\label{eq:conv_rate_W1}
    \D{\reg}{p^\dagger}(u_n,\Jminsol)=O\left(\frac{\delta_n}{\alpha_n}+\frac{r(\alpha_n)}{\alpha_n}+\eta_n\right),
\end{equation}
where $p^\dagger = -B^*\mu^\dagger$ is the subgradient from Assumption~\ref{ass:sc}.

Under the additional assumption that $A,A^{l,u}$ are bounded from as operators $\X \to L^\infty(\Omega) \subset \M(\Omega)$, we get the same rate for the symmetric Bregman distance~$\Dsymm{\reg} (u_n,\Jminsol)$ (cf. Theorem~\ref{thm:conv_rate_symm}).
\end{theorem}
\begin{proof}
Taking $h = \alpha_n E^*\mu^\dagger$ and $\nu = f_n$ in~\eqref{eq:H_star_phi_div}, and using~\eqref{eq:phi_star}, we get
\begin{eqnarray*}
    \adj{\fid}(\alpha_n E^*\musc \mid f_n) - \sp{\alpha_n E^*\musc,\b f} & = & \sp{\phi^*(\alpha_n E^* \mu^\dagger),f_n} - \sp{\alpha_n E^*\musc,\b f} \\
    & = & \sp{\adj{\phi}(\alpha_n E^*\musc)+\adj{\phi}(-\alpha_n E^*\musc), f_n} \\ 
    && \quad + \sp{-\alpha_n E^*\musc,\b f}-\sp{\adj{\phi}(-\alpha_n E^*\musc),f_n} \\
    & \leq & \sp{\adj{\phi}(\alpha_n E^*\musc)+\adj{\phi}(-\alpha_n E^*\musc), f_n} + d_\phi(\b f,f_n) \\
    & \leq & \sp{\adj{\phi}(\alpha_n E^*\musc)+\adj{\phi}(-\alpha_n E^*\musc), f_n} + \delta_n \\
    &=& \sp{r(\alpha_n E^*\musc)+r(-\alpha_n E^*\musc), f_n} + \delta_n,
\end{eqnarray*}
and in combination with~\eqref{eq:Breg_dist_one-sided} this yields the assertion.
\end{proof}

\paragraph{KL-divergence.}
Here $\phi(x)= x \log(x) - (x-1)$, $\phi^*(x) = e^{x}-1 = x + r(x)$ with $r(x)=x^2/2+x^3/6\dots$ and we get that
\begin{equation}
    D^{p^\dagger}_\reg (u_n, \Jminsol) 
     =\bigO\left(\frac{\delta_n}{\alpha_n} + \alpha_n + \eta_n \right).
\end{equation}
which coincides with~\cite{Benning_Burger_general_fid_2011} in the case of an exact operator. For $\alpha_n \sim (\delta_n)^{\frac12}$ we get the optimal rate
\begin{equation}\label{eq:phi_div_optim_rate}
    D^{p^\dagger}_\reg (u_n, \Jminsol) 
     =\bigO\left((\delta_n)^{\frac12} + \eta_n \right).
\end{equation}

\paragraph{$\chi^2$-divergence.}
Here $\phi(x) = (x-1)^2$ and $\phi^*(x) = x + \frac{x^2}{4}$. Again, 
\begin{equation}
    D^{p^\dagger}_\reg (u_n, \Jminsol)
    =\bigO\left(\frac{\delta_n}{\alpha_n} + \alpha_n + \eta_n \right)
\end{equation}
and the optimal rate coincides with~\eqref{eq:phi_div_optim_rate}.

\paragraph{Squared Hellinger distance.}
Here $\phi(x) = (\sqrt{x}-1)^2$, $\phi^*(x) = \frac{x}{x-1} \approx x + x^2 +...$ and we get
\begin{equation}
    D^{p^\dagger}_\reg (u_n, \Jminsol)  
    =\bigO\left(\frac{\delta_n}{\alpha_n} + \alpha_n + \eta_n \right)
\end{equation}
and the optimal rate coincides with~\eqref{eq:phi_div_optim_rate}.

\paragraph{Total variation.}
For the total variation (of measures) we have $\phi(x) = \frac12 \abs{x-1}$ and 
\begin{equation*}
\phi^*(x) = \begin{cases}
x, \quad & \abs{x} \leq \frac12, \\
\infty, \quad & \text{otherwise}.
\end{cases}
\end{equation*}
Then for any $\alpha_n = const$ such that $\norm{\alpha_n E^* \mu^\dagger}_\infty \leq \frac12$ we get that
\begin{equation}
    D^{p^\dagger}_\reg (u_n, \Jminsol)  
    =\bigO\left(\delta_n + \eta_n \right).
\end{equation}

\revtwo{
\begin{remark}[Poisson noise]
The main motivation for the use of the Kullback-Leibler divergence as a fidelity term is the modelling of Poisson noise~\cite{Le_Poisson_noise:2007}. If $t$ denotes the exposure time, the measured data can be assumed to be generated by a Poisson process with intensity $t\bar f$. In this case, the upper bound on the error in the Kullback-Leilbler divergence is given by~\cite{hohage2013poisson}
\begin{eqnarray*}
	\fid(\bar f \mid f_n) \leq \frac{1}{\sqrt{t_n}}.
\end{eqnarray*}
While in the deterministic setting, this estimate is sufficient to obtain convergence rates, the statistical setting requires further assumptions, in particular some concentration inequalities~\cite{werner2012poisson, hohage2013poisson, hohage2016poisson}.
\end{remark}
}

\subsubsection{Strongly coercive fidelity terms}
\label{sss:strongly_coercive}
\begin{theorem}\label{thm:conv_rates_coercive}
Suppose that the fidelity function $\fid$ is coercive in the following sense
\begin{equation}\label{eq:fid_coerc}
    \frac{C}{\lambda}\norm{v-f}_\Y^\lambda \leq \fid(v \mid f)
\end{equation}
for all $v,f \in \Y$, where $\lambda \geq 1$ and $C>0$ are constants (we will assume with loss of generality that $C = 1$). Then under the assumptions of Theorem~\ref{thm:conv_rate_one_sided} the following convergence rates hold
\begin{eqnarray*}
    \D{\reg}{p^\dagger} (u_n,\Jminsol) 
    & = & 
    \begin{cases}
    \bigO\left(\frac{\delta_n}{\alpha_n} +  \alpha_n^{\frac{1}{\lambda-1}}  + \delta_n^{\frac1\lambda} + \eta_n\right), \quad  & \lambda>1, \\
    \bigO\left(\frac{\delta_n}{\alpha_n} + \delta_n + \eta_n\right), \quad & \lambda = 1,
    \end{cases}
\end{eqnarray*}
where $p^\dagger = -B^*\mu^\dagger$ is the subgradient from Assumption~\ref{ass:sc}. If $\alpha_n$ is chosen such that $\alpha_n \sim (\delta_n)^{\frac{\lambda-1}{\lambda}}$ if $\lambda>1$ and $\alpha_n = const \leq \frac{1}{\norm{\adj{E}\musc}}$ if $\lambda = 1$, we get the optimal rate
\begin{eqnarray*}
    \D{\reg}{p^\dagger} (u_n,\Jminsol) 
    & = & 
    \bigO\left(\delta_n^{\frac{1}{\lambda}} + \eta_n\right).
\end{eqnarray*}
If $\Y$ is an AM-space (cf. Theorem~\ref{thm:conv_rate_symm}), the same rate holds for the symmetric Bregman distance~$\Dsymm{\reg} (u_n,\Jminsol)$.
\end{theorem}
\begin{proof}
Since convex conjugation is order-reversing, from~\eqref{eq:fid_coerc} we obtain that for any $q \in \Y^*$ (we will drop the subscripts $\Y$ and $\Y^*$ after the norms to simplify notation)
\begin{eqnarray*}
    \fid^*(q \mid f) & \leq & \left(\frac{1}{\lambda}\norm{\cdot - f}^\lambda\right)^*(q) = \left(\frac{1}{\lambda}\norm{\cdot}^\lambda\right)^*(q) + \sp{f,q} \\
    & = & \begin{cases}
    \frac{1}{\lambda^*}\norm{q}^{\lambda^*}  + \sp{f,q} \quad & \lambda >1, \\
    \charf_{\norm{\cdot} \leq 1}(q)  + \sp{f,q} \quad & \lambda = 1,
    \end{cases}
\end{eqnarray*}
where $\lambda^* = \frac{\lambda}{\lambda-1}$. We will consider the cases $\lambda >1$ and $\lambda = 1$ separately.

Let $\lambda > 1$. Then from Theorem~\ref{thm:conv_rate_one_sided} we obtain
\begin{eqnarray*}
    \D{\reg}{p^\dagger} (u_n,\Jminsol) & \leq & \frac{\delta_n}{\alpha_n} + \frac{1}{\alpha_n}\left(
    \frac{1}{\lambda^*}\norm{\alpha_n \adj{E}\musc}^{\lambda^*}  + \sp{\alpha_n \adj{E}\musc,f_n}
    - \sp{\alpha_n \adj{E}\musc, \b f}\right) + C\eta_n \\
    & = & \frac{\delta_n}{\alpha_n} +\frac{1}{\alpha_n}\left(\frac{\alpha_n^{\lambda^*}}{\lambda^*}\norm{\adj{E}\musc}^{\lambda^*} +  \alpha_n \sp{\adj{E}\musc,f_n - \b f}\right)  + C\eta_n\\
    & = & \frac{\delta_n}{\alpha_n} + \frac{\alpha_n^{\lambda^*-1}}{\lambda^*}\norm{\adj{E}\musc}^{\lambda^*} +  \sp{\adj{E}\musc,f_n - \b f} + C\eta_n.
\end{eqnarray*}
Condition $\eqref{eq:fid_coerc}$ implies that $\norm{f_n-\b f} \leq C\delta_n^{\frac{1}{\lambda}}$. Hence, using the Cauchy-Schwarz inequality, we obtain
\begin{eqnarray*}
    \D{\reg}{p^\dagger} (u_n,\Jminsol) & \leq & \frac{\delta_n}{\alpha_n} + C\frac{\alpha_n^{\lambda^*-1}}{\lambda^*}\norm{\adj{E}\musc}^{\lambda^*} +  \norm{\adj{E}\musc}\norm{f_n - \b f} + C\eta_n \\
    & \leq & \frac{\delta_n}{\alpha_n} + C\frac{\alpha_n^{\lambda^*-1}}{\lambda^*}\norm{\adj{E}\musc}^{\lambda^*} +  \norm{\adj{E}\musc}\delta_n^{\frac1\lambda} + C\eta_n
    \\
    & = & \bigO\left(\frac{\delta_n}{\alpha_n} +  \alpha_n^{\frac{1}{\lambda-1}}  + \delta_n^{\frac1\lambda} + \eta_n\right).
\end{eqnarray*}
Let now $\lambda = 1$. Then for sufficiently small $\alpha_n \leq \frac{1}{\norm{\adj{E}\musc}}$ we obtain from Theorem~\ref{thm:conv_rate_one_sided}
\begin{eqnarray*}
    \D{\reg}{p^\dagger} (u_n,\Jminsol) & \leq & \frac{\delta_n}{\alpha_n} +  \sp{\adj{E}\musc,f_n - \b f} + C\eta_n \\
    & \leq & \frac{\delta_n}{\alpha_n} +  \norm{\adj{E}\musc}\norm{f_n - \b f} + C\eta_n \\
    & \leq & \frac{\delta_n}{\alpha_n} +  C\delta_n  + C\eta_n.
\end{eqnarray*}
For a sufficiently small but fixed $\alpha_n$ we get that
\begin{eqnarray*}
    \D{\reg}{p^\dagger} (u_n,\Jminsol) & = & \bigO\left(\delta_n + \eta_n\right).
\end{eqnarray*}
\end{proof}

\begin{remark}
The value $\frac{1}{\norm{\adj{E}\musc}}$ matches the exact penalisation parameter in regularisation with one-homogeneous fidelity terms (e.g.~\cite{Burger_Osher:2004,Benning_Burger_general_fid_2011,Bungert_Burger:2019}).
Exact penalisation means that the regularisation parameters $\alpha_n$ do not have to be sent to zero in order to obtain convergence in the Bregman distance. 
It is observed if the subdifferential $\partial\fid(\cdot \mid \b f)\vert_{\b f}$ is no singleton.
\end{remark}

\begin{example}[Powers of norms]
Theorem~\ref{thm:conv_rates_coercive} obviously applies if the fidelity function is given by a power of the norm, i.e.
\begin{equation*}
    \fid(v \mid f) = \frac{1}{\lambda}\norm{v-f}^\lambda, \quad \lambda \geq 1.
\end{equation*} 
This covers important cases such as the squared $L^2$ norm fidelity which is used to model Gaussian noise and the $L^1$ norm fidelity which is often used to model salt-and-pepper noise~\cite{Nikolova_salt_pepper:2005}.
\end{example}

\begin{example}[Wasserstein distances]\label{ex:W_1}
For any $p \geq 1$, the $p$-Wasserstein distance between two probability measures $\rho,\nu\in\P(\Omega)$
is defined as follows (cf.~\cite{santambrogio2015optimal})
\begin{equation*}
    W_p(\rho, \nu) \defeq \left(\inf_{\gamma\in\Pi(\rho,\nu)}\int_\Omega\int_\Omega|x-y|^p\d\gamma(x,y) \right)^\frac1p,
\end{equation*}
where $\Pi(\rho,\nu)$ is the space of probability measures on $\Omega\times\Omega$ with marginals $\rho$ and $\nu$.

Let the data space $\Y = \KR(\Omega)$ be the closure of the space of Radon measures $\M(\Omega)$ with respect to the Kantorovich-Rubinstein norm 
\begin{equation*}
\norm{\mu}_{\KR} \defeq    \sup \left\{ \int g\, d\mu \colon \Lip(g) \leq 1, \,\, \norm{g}_\infty \leq 1 \right\},
\end{equation*}
where $\Lip$ denotes the Lipschitz constant~\cite{Bogachev:2007}.
Obviously it holds $\norm{\mu}_{\KR}\leq \abs{\mu}(\Omega)$ for all $\mu\in\M(\Omega)$ and $\norm{\mu}_{\KR}\geq\abs{\mu}(\Omega)$ if $\mu \geq 0$ by choosing $g\equiv 1$ (it is known that the positive cone $\M_+(\Omega)$, and hence also the set of probability measures $\P(\Omega)$, is closed in the KR norm~\cite[Thm.~8.9.4]{Bogachev:2007}).
For any $v \in \Y$ and a probability measure $f \in \P(\Omega)$ we let

\begin{equation}\label{eq:fid_Wp}
    \fid(v \mid f) \defeq 
    \begin{cases}
    W_p^p(v \mid f) , \quad &\text{if }v\in\P(\Omega),\\
    \infty,\quad &\text{else.}
    \end{cases}
\end{equation}
It is well known that for any two probability measures $\rho,\nu\in\P(\Omega)$ 
\begin{equation*}
     W_1(\rho, \nu) = \norm{\rho-\nu}_{KR}.
\end{equation*}
It is also known that for any $q \leq p$ and any two probability measures $\rho,\nu\in\P(\Omega)$, the following relation holds~\cite{santambrogio2015optimal}
\begin{equation*}
    W_q(\rho, \nu) \leq W_p(\rho, \nu).
\end{equation*}
Hence, the data term defined in~\eqref{eq:fid_Wp} satisfies
\begin{equation*}
     \norm{v-f}_{\KR}^p \leq \fid(v \mid f),
\end{equation*}
i.e. it is strongly coercive on $\KR(\Omega)$. Note that it is not strongly coercive on $\M(\Omega)$ equipped with the total variation norm.

Hence, using Theorem~\ref{thm:conv_rates_coercive} we get the following optimal rate
\begin{equation*}
    \D{\reg}{p^\dagger} (u_n,\Jminsol) 
     = \bigO\left(\delta_n^{\frac{1}{p}} + \eta_n\right).
\end{equation*}
\end{example}


\subsubsection{Characteristic function of a norm ball}
Let the fidelity function be as follows
\begin{equation}\label{eq:fid_fun_charf}
    \fid(v \mid f) = \charf_{\norm{\cdot} \leq \delta} (v-f).
\end{equation}
This type of fidelity functions corresponds to the so-called residual method~\cite{TGSYag,GrasmairHalmeierScherzer2011} and allows one to explicitly use the noise level $\delta$ in the reconstruction (another way of doing so is the  discrepancy principle, see Section~\ref{sec:discr_pr}). It is clear that
\begin{equation*}
    \fid(v \mid f) \leq \delta \quad \Leftrightarrow \quad \norm{v - f} \leq \delta.
\end{equation*}
With this particular fidelity function the parameter $\alpha$ does not have any effect on the solutions of~\eqref{eq:primal}, hence with no loss of generality we will assume $\alpha_n = const$ for all $n$.

The coercivity assumption~\eqref{eq:fid_coerc} is not satisfied for this fidelity function (it is only \emph{weakly} coercive, i.e. $\norm{v-f} \to \infty$ implies $\fid(v \mid f) \to \infty$) and Theorem~\ref{thm:conv_rates_coercive} does not apply. 
\begin{theorem}\label{thm:conv_rate_charf}
Let the fidelity function be as defined in~\eqref{eq:fid_fun_charf}. Then under the assumptions of Theorem~\ref{thm:conv_rate_one_sided} the following convergence rate holds
\begin{equation}
    D^{p^\dagger}_\reg (u_n, \Jminsol) 
    = \bigO\left(\delta_n + \eta_n\right),
\end{equation}
where $p^\dagger = -B^*\mu^\dagger$ is the subgradient from Assumption~\ref{ass:sc}.

If $\Y$ is an AM-space (cf. Theorem~\ref{thm:conv_rate_symm}), the same rate holds for the symmetric Bregman distance~$\Dsymm{\reg} (u_n,\Jminsol)$.
\end{theorem}
\begin{proof}
Taking the convex conjugate of $\fid(\cdot \mid f)$ defined in~\eqref{eq:fid_fun_charf}, we get
\begin{eqnarray*}
    \adj\fid(q \mid f) &=& \sup_{v \colon \norm{v-f} \leq \delta} \sp{q,v} = \sup_{v \colon \norm{v-f} \leq \delta} \sp{q,v-f} + \sp{q,f} \\ 
    &\leq& \sup_{v \colon \norm{v-f} \leq \delta}\norm{q}\norm{v-f}  + \sp{q,f} \leq \delta\norm{q}  + \sp{q,f}.
\end{eqnarray*}
Hence,
\begin{eqnarray*}
    \adj\fid(\alpha_n E^*\mu^\dagger \mid f_n) - \sp{\alpha_n E^*\mu^\dagger, \b f} &\leq& \delta_n \alpha_n \norm{E^* \mu^\dagger} + \sp{\alpha_n E^*\mu^\dagger, f_n-\b f} \\
    &\leq& 2\delta_n \alpha_n \norm{E^* \mu^\dagger}
\end{eqnarray*}
since $\norm{f_n - \b f} \leq \delta_n$. Plugging this into the estimate in Theorem~\ref{thm:conv_rate_one_sided} (resp. Theorem~\ref{thm:conv_rate_symm}) and remembering that $\alpha_n = const$ for all $n$, we get the assertion.
\end{proof}

\subsubsection{Sum of fidelities}
Having studied a plethora of explicit examples of fidelity functions, we now turn to combinations of several fidelities, each of which can be studied as above.
Let us assume that $\fid$ is the sum of two other fidelity functions $\fid_1$ and $\fid_2$, i.e.,
\begin{equation}\label{eq:fidelity_sum}
    \fid(v \mid f) = \fid_1(v \mid f)+\fid_2( v \mid f).
\end{equation}
Such fidelities were studied e.g. in \cite{holler2018coupled} and allow to simultaneously handle data from different modalities.
Furthermore, in \cite{hintermuller2013subspace,langer2017automated,yue2014locally} fidelites of $L^1+L^2$-type were analysed and used for image restoration in the presence of mixed Gaussian and impulse noise.

If $\fid_1$ and $\fid_2$ are proper, it holds
\begin{equation}\label{ineq:conj_sum}
    \fid^*(q \mid f)\leq \inf_{r\in\Y^*}\left\lbrace\fid_1^*(r \mid f) + \fid_2^*(q-r \mid f)\right\rbrace=:\left(\fid_1^*(\cdot \mid f)\infconv\fid_2^*(\cdot \mid f)\right)(q),
\end{equation}
where the term on the right hand side is the so-called infimal convolution of $\fid_1$ and $\fid_2$.
Let us assume that we have estimates of the form 
\begin{equation}\label{ineq:rate_functions}
    \fid_i^*(\alpha_n E^* \mu^\dagger \mid f_n) - \langle \alpha_n E^* \mu^\dagger \mid \b f) \leq R_i(\alpha_n, \fid_i(\b f \mid f_n)), \quad i = 1,2,
\end{equation}
for each of the fidelities. 
The functions $R_i$ are assumed to be non-decreasing in both arguments and we set $R_i(\alpha,\cdot)=\infty$ for $\alpha<0$.
Combining \eqref{ineq:conj_sum} and \eqref{ineq:rate_functions} we obtain
\begin{eqnarray*}
    &&\fid^*(\alpha_n E^* \mu^\dagger \mid f_n) - \langle \alpha_n E^* \mu^\dagger, \b f \rangle \\
    &=&\inf_{w\in\Y}\left\lbrace\fid_1^*(w \mid f_n) + \fid_2^*(\alpha_n E^* \mu^\dagger - w \mid f_n)\right\rbrace - \langle \alpha_n E^* \mu^\dagger, \b f \rangle \\
    &\leq& \inf_{\lambda\in[0,1]} \fid_1^*(\lambda\alpha_n E^* \mu^\dagger \mid f_n ) - \langle \lambda\alpha_n E^* \mu^\dagger, \b f \rangle + \fid_2^*((1-\lambda)\alpha_n E^*\mu^\dagger \mid f_n) - \langle (1-\lambda)\alpha_n E^* \mu^\dagger, \b f \rangle  \\
    &\leq& \inf_{\lambda\in[0,1]} R_1(\lambda\alpha_n, \fid_1(\b f \mid f_n)) + R_2((1-\lambda)\alpha_n,\fid_2(\b f \mid f_n)) \\
    &\leq& \inf_{\lambda\in[0,1]} R_1(\lambda\alpha_n, \delta_n) + R_2((1-\lambda)\alpha_n,\delta_n) \\
    &=&  \left(R_1(\cdot, \delta_n) \infconv R_2(\cdot, \delta_n)\right)(\alpha_n),
\end{eqnarray*}
where we used the monotonicity properties of $R_i$ in the last two steps.
This shows that the convergence rate for $\fid$ can be estimated by the infimal convolution of the rates associated to $\fid_1$ and $\fid_2$, i.e.
\begin{equation}
    \D{\reg}{p^\dagger}(u_n,\Jminsol) = \bigO\left[ \left(R_1(\cdot,\delta_n) \infconv R_2(\cdot,\delta_n)\right) (\alpha_n) +\eta_n \right].
\end{equation}
If $\Y$ is an AM-space (cf. Theorem~\ref{thm:conv_rate_symm}), the same rate holds for the symmetric Bregman distance~$\Dsymm{\reg} (u_n,\Jminsol)$.

\subsubsection{Infimal convolution of fidelities}
Let us consider the case that $\fid$ is given by the infimal convolution of two other fidelities $\fid_1$ and $\fid_2$
\begin{eqnarray}\label{eq:fidelity_inf_conv}
    \fid(v \mid f)  =  \inf_{w \in \Y}  \fid_1(w \mid 0)+\fid_2( v-w \mid f) = (\fid_1(\cdot \mid 0) \infconv \fid_2(\cdot \mid f))(v)
\end{eqnarray}
Such fidelities are also chosen for the removal of mixed noise in image restoration (see e.g.~\cite{aggarwal2016hyperspectral} for an application to hyperspectral unmixing and~\cite{calatroni2017infimal} and the references therein for image denoising with mixtures of Gaussian, impulse, and Poisson noise).
Since the infimal convolution \emph{optimally} decomposes $v$ into a noise part $w$, which is small in $\fid_1$, and a residual $v-w$, which is close to the data $f$ in $\fid_2$, such fidelities are more suitable for this purpose than the plain sum of fidelities, studied in the previous section.
By standard calculus for infimal convolutions, if $\fid_1$ and $\fid_2$ are proper, it holds
\begin{equation}\label{eq:conj_inf_conv}
    \fid^*(q \mid f) = \fid_1^*(q \mid 0) + \fid_2^*( q \mid f).
\end{equation}
Furthermore, under the hypothesis that $\fid_1$ is coercive, $\fid_2$ is bounded from below, and both are weakly-* lower semicontinuous convex functions, it holds that $\fid$ is weakly-* lower semicontinuous, proper, and exact~(see~\cite{stromberg1994study} for the statement and~\cite{Bauschke_Combettes:2011} for a proof on Hilbert spaces which generalises to Banach spaces).
The latter means that the infimum in the definition of $\fid$ is attained.
In particular, there are $\b g,\b h \in \Y$ such that $\b f = \b g + \b h$ and 
\begin{equation}\label{eq:delta_inf_conv}
\delta_n=\fid(\b f, f_n)=\fid_1(\b g \mid 0)+\fid_2(\b h \mid f_n).
\end{equation}
Furthermore, from \eqref{eq:conj_inf_conv} we get
\begin{eqnarray*}
    &&\fid^*(\alpha_n E^* \mu^\dagger \mid f_n) - \langle \alpha_n E^* \mu^\dagger, \b f \rangle \\
    &=& \fid_1^*(\alpha_n E^* \mu^\dagger \mid 0) + \fid_2^*(\alpha_n E^* \mu^\dagger \mid f_n) - \langle \alpha_n E^* \mu^\dagger, \b f \rangle \\
    &=& \left(\fid_1^*(\alpha_n E^* \mu^\dagger \mid 0) - \langle \alpha_n E^* \mu^\dagger, \b g \rangle\right) + \left(\fid_2^*(\alpha_n  E^* \mu^\dagger \mid f_n) - \langle \alpha_n E^* \mu^\dagger, \b h\rangle\right).
\end{eqnarray*}
Consequently, we have to estimate the two terms in brackets which only depend on the individual fidelites $\fid_1$ and $\fid_2$.
In all the examples studied above, such estimates are available.
Using the functions $R_i$ defined in \eqref{ineq:rate_functions} above together with \eqref{eq:delta_inf_conv}, we can estimate
\begin{eqnarray*}
    &&\fid^*(\alpha_n E^* \mu^\dagger \mid f_n) - \langle \alpha_n E^* \mu^\dagger, \b f \rangle \\
    &\leq& R_1(\alpha_n, \fid_1(\b g \mid 0)) + R_2(\alpha_n, \fid_2(\b h \mid f_n)) \\
    &\leq& R_1(\alpha_n, \delta_n) + R_2(\alpha_n, \delta_n).
\end{eqnarray*}
Hence, we get the statement that the rate of convergence of a infimal convolution of fidelities can be estimated by the sum of the individual rates associated to $\fid_1$ and $\fid_2$, i.e.
\begin{equation}
    \D{\reg}{p^\dagger}(u_n, \Jminsol) = \bigO\left(R_1(\alpha_n, \delta_n) + R_2(\alpha_n, \delta_n)+\eta_n \right).
\end{equation}
This is in contrast to the rate of a sum of fidelities being given by the infimal convolution of the rates, as shown in the previous section.

If $\Y$ is an AM-space (cf. Theorem~\ref{thm:conv_rate_symm}), the same rate holds for the symmetric Bregman distance~$\Dsymm{\reg} (u_n,\Jminsol)$.

\section{Discrepancy Principle}\label{sec:discr_pr}
When the operator is known exactly, Morozov's discrepancy principle~\cite{Morozov:1966,engl:1996} can be used to select the regularisation parameter $\alpha_n$. In the case of a squared norm fidelity $\fid(v \mid f) = \norm{v-f}^2$ this amounts to selecting $\alpha_n$ such that
\begin{equation}\label{eq:discr_pr1}
    \alpha_n = \sup\{\alpha>0 \colon \norm{A u_n^{\alpha_n} - f_n}^2 \leq \tau\delta_n \},
\end{equation}
where $u_n^{\alpha_n}$ is the regularised solution corresponding the the regularisation parameter $\alpha_n$ and $\tau>1$ is a parameter. Here we assume that $\norm{\b f - f_n}^2 \leq \delta_n$ (and not $\norm{\b f - f_n}^2 \leq \delta_n^2$) to be consistent with our earlier notation. Convergence rates for this choice of $\alpha_n$ in the case of an exact operator and an arbitrary convex regularisation functional were obtained in~\cite{Bonesky:2008}. For the data fidelity given by the Kullback-Leibler divergence, the discrepancy principle is studied in~\cite{Sixou:2018}.

In the case of an imperfect operator, the discrepancy principle needs to be modified. 
When the operator error is measured using the operator norm, i.e. one assumes that an approximate operator $A_h$ is available such that
\begin{equation*}
    \norm{A-A_{h_n}}_{\mathcal L(\X, \Y)} \leq h_n,
\end{equation*}
one can choose $\alpha_n$ as follows~\cite{TGSYag} (in the case of a squared norm fidelity in the Hilbert space setting)
\begin{equation}\label{eq:discr_pr2}
    \alpha_n = \sup\{\alpha>0 \colon \norm{Au_n^{\alpha_n}-f_n}^2 = (\sqrt{\delta_n} + h_n \norm{u_n^{\alpha_n}})^2\}.
\end{equation}
If the fidelity term is not based on a norm and does not satisfy the triangle inequality, such a generalisation is not available. 

Since in our case the operator error is explicitly accounted for  through the constraints in~\eqref{eq:primal}, we can use the discrepancy principle in its original form~\eqref{eq:discr_pr1} with an \emph{arbitrary} fidelity term. We will choose $\alpha_n$ such that
\begin{equation}\label{eq:discr_pr3}
    \alpha_n = \sup\{\alpha>0 \colon  \fid(v_n^{\alpha} \mid f_n) \leq \tau\delta_n \},
\end{equation}
where $v_n^{\alpha_n}$ solves~\eqref{eq:primal} with the regularisation parameter $\alpha_n$ and $\tau>1$ is a parameter.

\begin{remark}
If the solution $v_n^{\alpha_n}$ is unique, then we have
\begin{equation}\label{eq:discr_pr_equality}
    \fid(v_n^{\alpha_n} \mid f_n) = \tau\delta_n.
\end{equation}
In case of non-uniqueness, we can always choose a solution $v_n^{\alpha_n}$ such that~\eqref{eq:discr_pr_equality} is satisfied, following the argument in \cite[Prop. 3.5 -- Rem. 3.8]{Anzengruber_Ramlau:2009} and using convexity of the objective function in~\eqref{eq:primal}.
\end{remark}


\subsection{Existence}
In this section we study well-posedness of the discrepancy principle, meaning that there is a regularisation parameter $\alpha_n$ which meets~\eqref{eq:discr_pr3}.
Let $(u^\alpha, v^\alpha)$ be a solution of~\eqref{eq:primal} corresponding to the parameter $\alpha>0$. Define the following functions:
\begin{equation}\label{eq:h_and_j}
    h(\alpha) \defeq \fid(v^\alpha \mid f_n), \quad j(\alpha) \defeq \reg(u^\alpha).
\end{equation}

\begin{lemma}\label{lem:mon}
The function $j(\alpha)$ is  monotone non-increasing and $h(\alpha)$ is monotone non-decreasing in $\alpha$. 

\end{lemma}
\begin{proof}
The proof is similar to \cite{Burger_Osher_TV_Zoo}.
\end{proof}
\begin{remark}
If either $\fid(\cdot \mid f_n)$ or $\reg(\cdot)$ is strictly convex, then $h(\alpha)$ and $j(\alpha)$ are indeed uniquely defined (the argument is similar to~\cite{Bungert_Burger:2019}). Otherwise the lemma applies to $\fid(v^\alpha \mid f_n)$ and $\reg(u^\alpha)$ for any solution $(u^\alpha, v^\alpha)$ of~\eqref{eq:primal}.
\end{remark}
\begin{remark}
Since $j$ and $h$ are monotone functions, they are in particular continuous for almost all values of $\alpha>0$.
\end{remark}
\begin{lemma}\label{lem:lsc_of_h_and_j}
Functions $h$ and $j$ defined in \eqref{eq:h_and_j} are lower semicontinuous.
\end{lemma}
\begin{proof}
We just sketch the proof. 
Letting $\alpha_k\to\alpha$, one can easily see that the corresponding solutions $(v_k,u_k)$ converge (up to a subsequence) weakly-* to $(v,u)$ which solve the problem for $\alpha$. 
Hence, by the lower semicontinuity of $\fid$ and $\reg$ the assertion follows.
\end{proof}

\begin{theorem}\label{thm:well-posedness_discrepancy}
Suppose that for all $n$
\begin{equation*}
    C\delta_n \leq \liminf_{\alpha\to\infty}\fid(v^\alpha \mid f_n)
\end{equation*}
for some constant $C>1$, which does not depend on $n$. 

Then the discrepancy principle~\eqref{eq:discr_pr3} is well-posed for all $\tau\in(1,C)$, i.e. there exists $\alpha_n > 0$ and a solution $(u^{\alpha_n},v^{\alpha_n})$ of~\eqref{eq:primal} corresponding to $\alpha = \alpha_n$ and $f = f_n$ such that~\eqref{eq:discr_pr3} is satisfied.
\end{theorem}
\begin{proof}
For every $\alpha>0$ because of the feasibility of $(\Jminsol,\b f)$ we get
\begin{equation*}
    \fid(v^\alpha \mid f_n) + \alpha \reg(u^\alpha) \leq \fid(\b f \mid f_n) + \alpha \reg(\Jminsol)
\end{equation*}
and in particular
\begin{equation*}
    h(\alpha)=\fid(v^\alpha \mid f_n) \leq \delta_n + \alpha \reg(\Jminsol),
\end{equation*}
for almost all $\alpha>0$. 
Letting $\alpha \downarrow 0 $ we obtain using the monotonicity of $h$ that
\begin{equation}\label{ineq:lower_limit_h}
    h(0+)\leq \delta_n.
\end{equation}
On the other hand, by assumption it holds
\begin{equation}\label{ineq:upper_limit_h}
    C\delta_n \leq \liminf_{\alpha\to\infty}\fid(v^\alpha \mid f_n).
\end{equation}
Hence, in light of \eqref{ineq:lower_limit_h},~\eqref{ineq:upper_limit_h}, and the monotonicity of $h$, there exists $\alpha_n >0 $ such that
\begin{equation*}
    h(\alpha) \leq \tau \delta_n, \quad \forall 0 < \alpha < \alpha_n,
\end{equation*}
and $\tau$ can be chosen in $(1,C)$.
Since $h$ is lower semicontinuous according to Lemma~\ref{lem:lsc_of_h_and_j}, we get that
\begin{equation*}
    \sup_{\alpha<\alpha_n}h(\alpha) \leq \tau \delta_n
\end{equation*}
which proves the assertion.
\end{proof}


\begin{remark}
The assumption of Theorem~\ref{thm:well-posedness_discrepancy} is rather weak.
For instance, if $\fid(0 \mid f_n)<\infty$, one can show that $v^\alpha\wsto 0$ as $\alpha\to\infty$. Hence, one can relax the assumption to $C\delta_n\leq\fid(0 \mid f_n)$ which, for $\delta_n$ sufficiently small, is fulfilled in many applications.
\end{remark}

 

\subsection{Convergence rates}
Our goal in this section is to obtain convergence rates similar to those in Theorems~\ref{thm:conv_rate_one_sided} (respectively Theorem~\ref{thm:conv_rate_symm}) for the parameter choice rule~\eqref{eq:discr_pr3}.  

\begin{lemma}\label{lem:discr_pr}
Let $\alpha_n$ be chosen according to~\eqref{eq:discr_pr3}. Then the following inequality holds for all $n$
\begin{equation}
    \reg(u_n^{\alpha_n}) \leq \reg(\Jminsol).
\end{equation}
If conditions of Theorem~\ref{thm:robinson_primal} are satisfied, then also the following inequality holds    
\begin{equation}
    \sp{\adj{E}\mu_n^{\alpha_n},\b f - v_n^{\alpha_n}} \leq 0.
\end{equation}
\end{lemma}
\begin{proof}
Comparing the value of the objective function in~\eqref{eq:primal} at the optimal solution $(u_n^{\alpha_n},v_n^{\alpha_n})$ and $(\Jminsol, \b f)$ and using~\eqref{eq:discr_pr3}, we get that
\begin{equation*}
     \tau\delta_n + \alpha_n \reg(u_n^{\alpha_n}) = \fid(v_n^{\alpha_n} \mid f_n) + \alpha_n \reg(u_n^{\alpha_n}) \leq \fid(\b f \mid f_n) + \alpha_n \reg(\Jminsol) \leq \delta_n + \alpha_n \reg(\Jminsol).
\end{equation*}
Since $\tau>1$, this yields the first inequality.

For the second one we use the Fenchel--Young inequality. Subtracting~\eqref{Fenchel_Young_est2} from~\eqref{Fenchel_Young_est1} we obtain
\begin{eqnarray*}
    \sp{\alpha_n \adj{E}\mu_n^{\alpha_n}, \b f - v_n^{\alpha_n}} & \leq & \fid(\b f \mid f_n) + \fid^*(\alpha_n \adj{E}\mu_n^{\alpha_n} \mid f_n) - \fid(v_n^{\alpha_n} \mid f_n) - \fid^*(\alpha_n \adj{E}\mu_n^{\alpha_n} \mid f_n) \\
    & \leq & \delta_n - \tau\delta_n \leq 0,
\end{eqnarray*}
which completes the proof.
\end{proof}
\begin{theorem}
Under Assumptions of Theorem~\ref{thm:primal_conv} and with $\alpha_n$ chosen according to~\eqref{eq:discr_pr3}, $u_n^{\alpha_n}$
converges weakly-* to a $\reg$-minimising solution of~\eqref{eq:Au=f}, i.e.
\begin{equation*}
    u_n^{\alpha_n} \wsto \Jminsol.
\end{equation*}
\end{theorem}
\begin{proof}
Since $\reg(u_n^{\alpha_n})$ is bounded uniformly in $n$ and  $\fid(v_n^{\alpha_n} \mid f_n) = \tau \delta_n \to 0$, we immediately get the desired result following the proof of Theorem~\ref{thm:primal_conv}.
\end{proof}

\begin{theorem}\label{thm:conv_rates_discr}
Let $\alpha_n$ be chosen according to~\eqref{eq:discr_pr3}. Then, under the assumptions of Theorem~\ref{thm:conv_rate_one_sided}, the following estimate holds for the one-sided Bregman distance between $u_n^{\alpha_n}$ and~$\Jminsol$
\begin{equation*}
    \D{\reg}{p^\dagger} (u_n^{\alpha_n},\Jminsol) \leq \sp{E^*\musc,v_n^{\alpha_n} - \b f} + C\eta_n,
\end{equation*}
where $p^\dagger = -\adj{B}\musc$ is the subgradient from Assumption~\ref{ass:sc}.
Under the assumptions of Theorem~\ref{thm:conv_rate_symm} the same estimate holds for the symmetric Bregman distance.
\end{theorem}
\begin{proof}
We start with the estimate~\eqref{eq:Breg_dist_one-sided_intermediate}. Using Lemma~\ref{lem:discr_pr}, we obtain
\begin{eqnarray*}
    \D{\reg}{p^\dagger}(u_n^{\alpha_n},\Jminsol) &\leq& \reg(\Jminsol) + \sp{E^*\musc,v_n^{\alpha_n}} + C\eta_n \\
    &=& \sp{-B^*\musc,\Jminsol} + \sp{E^*\musc,v_n^{\alpha_n}} + C\eta_n \\
    &=& \sp{-\musc,E \b f} + \sp{E^*\musc,v_n^{\alpha_n}} + C\eta_n \\
    &=& \sp{E^*\musc,v_n^{\alpha_n} - \b f} + C\eta_n,
\end{eqnarray*}
which yields the first assertion. For the second assertion, we use~\eqref{Dsymm_est1} and Lemma~\ref{lem:discr_pr} and obtain
\begin{eqnarray*}
    \Dsymm{\reg} (u_n,\Jminsol) & \leq & \sp{\adj{E}\mu_n, \b f - v_n^{\alpha_n}} - \sp{\adj{E}\musc, \b f - v_n^{\alpha_n}} + C\eta_n \\
    & \leq & \sp{\adj{E}\musc, v_n^{\alpha_n} - \b f} + C\eta_n.
\end{eqnarray*}
\end{proof}




\paragraph{Strongly coercive fidelities.} For a strongly coercive fidelity terms such that \eqref{eq:fid_coerc} holds, we immediately get, using the Cauchy-Schwarz inequality, that
\begin{eqnarray*}
    \D{\reg}{p^\dagger} (u_n^{\alpha_n},\Jminsol) &\leq&  \norm{E^*\musc} \norm{v_n^{\alpha_n} - \b f} + C\eta_n \\
    &\leq& \norm{E^*\musc} (\norm{v_n^{\alpha_n} - f_n} + \norm{f_n - \b f}) + C\eta_n \\
    &\leq& \norm{E^*\musc} (\fid(v_n^{\alpha_n} \mid f_n) +  \fid(\b f \mid f_n))^{\frac1\lambda}  + C\eta_n
\end{eqnarray*}
and therefore we get the following rate
\begin{equation*}
    \D{\reg}{p^\dagger}(u_n^{\alpha_n},\Jminsol) = \bigO\left(\delta_n^{\frac1\lambda} + \eta_n\right),
\end{equation*}
which coincides with the optimal rate in Theorem~\ref{thm:conv_rates_coercive}.

\paragraph{$\phi$-divergences.} For any $\phi$-divergence that satisfies Pinsker's inequality~\cite{Saso_Verdu:2016} with exponent $\lambda$
\begin{equation*}
    \norm{v-f}^\lambda \leq C \fid(v \mid f),
\end{equation*}
where $v,f \in \P(\Omega)$, we have the same situation as above. In particular, for the Kullback-Leibler divergence, the $\chi^2$-divergence an the squared Hellinger distance $\lambda=2$ and 
\begin{eqnarray*}
    \D{\reg}{p^\dagger} (u_n^{\alpha_n},\Jminsol) 
    &=& \bigO(\sqrt{\delta_n} + \eta_n),
\end{eqnarray*}
which coincides with the optimal rate~\eqref{eq:phi_div_optim_rate}. 

We summarise all convergence rates for obtained in this paper in Table~\ref{tab:rates}.

\begin{table}[htp!]
    \centering
    \begin{tabular}{c||l|l|l}
         Fidelity & A priori rate & Optimal rate & Discr. principle \\
         \hhline{=#=|=|=}
         \begin{tabular}{@{}c@{}}
              KL- and $\chi^2$-divergences,  \\
              sq. Hellinger distance
         \end{tabular}   
         & $\bigO\left(\frac{\delta}{\alpha} + \alpha +\eta\right)$ & $\bigO(\sqrt{\delta}+\eta)$ & $\bigO(\sqrt{\delta}+\eta)$ \\
         \hline
         Total Variation & $\bigO(\delta+\eta)$  & 
         $\bigO(\delta+\eta)$ & $\bigO(\delta+\eta)$ \\
         \hline
         Wasserstein-$p$ distance & 
         \begin{tabular}{@{}lc}
              $\bigO\left(\frac{\delta}{\alpha} + \alpha^\frac{1}{p-1} + \delta^\frac{1}{p} +\eta\right)$, &$p>1$  \\
              $\bigO(\delta+\eta)$, &$p=1$ 
         \end{tabular}
         & $\bigO(\delta^\frac1p + \eta)$ & $\bigO(\delta^\frac1p + \eta)$ \\ 
         \hline
         \begin{tabular}{@{}c@{}}
         Characteristic function \\ of a norm ball 
         \end{tabular}
         & $\bigO(\delta+\eta)$ & $\bigO(\delta+\eta)$ & $\bigO(\delta+\eta)$ \\
         \hline
         \begin{tabular}{@{}c@{}}
         $\lambda$-strongly coercive \\ fidelities 
         \end{tabular}
         & 
         \begin{tabular}{@{}lc}
              $\bigO\left(\frac{\delta}{\alpha} + \alpha^\frac{1}{\lambda-1} + \delta^\frac{1}{\lambda} +\eta\right)$, &$\lambda>1$  \\
              $\bigO(\delta+\eta)$, &$\lambda=1$
         \end{tabular}
         & $\bigO(\delta^\frac{1}{\lambda} + \eta)$ & $\bigO(\delta^\frac{1}{\lambda} + \eta)$
    \end{tabular}
    \caption{\revthree{Summary of convergence rates for different fidelities in terms of the data error $\delta$, the operator error $\eta$ and the regularisation parameter $\alpha$. Whenever $\alpha$ is absent in the a priori rate, exact penalisation occurs and the rate is independent of $\alpha$ as long as it is smaller than a fixed constant. Optimal rates correspond to an optimal choice of $\alpha$ in the a priori rate.}}
    \label{tab:rates}
\end{table}

\section{Conclusions}
In this work we have proven convergence rates in Bregman distances for variational regularisation in Banach lattices for problems with imperfect forward operators and general fidelity functions. Our results apply to many classes of fidelity functions and recover known convergence rates for norm-type fidelities and the Kullback-Leibler divergence in the case of exact operators.
In addition, we have derived convergence rates for sums and infimal convolutions of fidelity functions, as used for mixed-noise removal. Furthermore, we have analysed an extension of Morozov's discrepancy principle to problems with operator errors in the Banach lattice setting, which does not rely on the triangle inequality and hence applies to a broader class of fidelity functions.


\section*{Acknowledgements}
This work was supported by the European Union’s Horizon 2020 research and innovation programme under the Marie Sk{\l}odowska-Curie grant agreement No. 777826 (NoMADS).

MB acknowledges support by the Bundesministerium f\"ur Bildung und Forschung under the project id 05M16PMB (MED4D).

YK acknowledges the support of the Royal Society (Newton International Fellowship NF170045), the EPSRC (Fellowship EP/V003615/1), the Cantab Capital Institute for the Mathematics of Information and the National Physical Laboratory. 

CBS acknowledges support from the Leverhulme Trust project on ‘Breaking the non-convexity barrier’, the Philip Leverhulme Prize, the EPSRC grants EP/S026045/1 and EP/T003553/1, the EPSRC Centre Nr. EP/N014588/1, the Wellcome Innovator Award RG98755, European Union Horizon 2020 research and innovation programmes under the Marie Sk{\l}odowska-Curie grant agreement No. 777826 NoMADS and No. 691070 CHiPS, the Cantab Capital Institute for the Mathematics of Information and the Alan Turing Institute.

\printbibliography

\appendix

\section{Banach lattices and duality}



The following definitions and results can be found, e.g., in~\cite{Meyer-Nieberg}.

Let $\U$ be a vector space and ``$\leq$'' be a partial order relation on $\U$ (i.e. a reflexive, antisymmetric and transitive binary relation). For all $x,y \in \U$ we write $x \geq y$ if $y \leq x$. 
The pair $(\U,\leq)$ is called an ordered vector space if the following conditions hold
\begin{eqnarray*}
    x \leq y &\implies& x+z \leq y+z \quad \forall z \in \U, \\
    x \leq y &\implies& ax \leq ay \quad \forall a \in \R_+.
\end{eqnarray*}

An ordered vector space $(\U,\leq)$ is called a vector lattice (or a Riesz space) if any two elements $x,y \in \U$ have a unique supremum $x \vee y$ and infimum $x \wedge y$. For any $x \in \U$ we define
\begin{equation*}
    x_+ \defeq x \vee 0, \quad x_- \defeq (-x)_+, \quad \abs{x} \defeq x_+ + x_-.
\end{equation*}
For any $x \in \U$ it holds that
\begin{equation*}
    x = x_+ - x_-.
\end{equation*}

Let $\norm{\cdot}$ be a norm on $\U$. The triple $(\U, \leq, \norm{\cdot})$ is called a Banach lattice if $(\U,\leq)$ is a vector lattice, $(\U,\norm{\cdot})$ is a Banach space (i.e. it is norm complete) and for all $x,y \in \U$
\begin{equation*}
    \abs{x} \geq \abs{y} \implies \norm{x} \geq \norm{y},
\end{equation*}
or equivalently that $\norm{x} \geq \norm{y}$ for any $x \geq y \geq 0$.

A linear operator $T$ acting between two vector lattices $\U_1$, $\U_2$ is called positive, and we write $T \geq 0$, if $u \geq 0$ implies $Tu \geq 0$ (the inequalities are understood in the sense of partial orders in $U_1$ and $\U_2$, respectively). A linear operator $T$ is called \emph{regular} if it can be written as a difference of two positive operators, $T=T_1-T_2$ with $T_{1,2} \geq 0$. The space of all regular operators $\U_1 \to \U_2$ is itself an ordered vector space with the following partial order
\begin{equation*}
    T_1 \geq T_2 \iff T_1 -T_2 \geq 0.
\end{equation*}

\begin{proposition}[{\cite[Prop. 1.3.5]{Meyer-Nieberg}}]
Let $\U_1$, $\U_2$ be Banach lattices. Then every regular operator $\U_1 \to \U_2$ is (norm) continuous.
\end{proposition}

The converse is in general false, i.e. not every continuous operator is regular. \revthree{However, in some settings this is true. We repeat Definition~\ref{def:AM-space} for readers' convenience.
\begin{definition}\label{def:AM-space-2}
A Banach lattice $\Y$ with norm $\norm{\cdot}$ is called an \emph{AM-space} (abstract maximum space) if
\begin{equation*}
    \norm{x \vee y}=\norm{x} \vee \norm{y},\quad \forall x,y \geq 0.
\end{equation*}
An element $\one \in \Y$ which meets
\begin{equation*}
    \one \geq 0,\quad \norm{\one} = 1,\quad \norm{x}\leq 1 \implies \abs{x} \leq \one,
\end{equation*}
is called \emph{unit} of $\Y$.
\end{definition}
\begin{definition}\label{def:AL-space}
A Banach lattice $\Y$ with norm $\norm{\cdot}$ is called an \emph{AL-space} (abstract Lebesgue space) if
\begin{equation*}
    \norm{x \vee y} = \norm{x} + \norm{y},\quad \forall x,y \geq 0.
\end{equation*}
\end{definition}
If either $\Y$ is an $AM$-space with an order unit or $\X$ is an $AL$-space, then every linear bounded operator is regular (under some additional conditions, see~\cite[Thm. 1.5.11]{Meyer-Nieberg} for a precise statement).
}

We need the following result.
\begin{lemma}[Partial order on the dual]\label{lem:dual_order}
Let $\U$ be a Banach space and $\U^*$ be its dual. 
If $(\U,\leq,\norm{\cdot})$ is a Banach lattice, then so is $\U^*$, equipped with the dual norm and the following partial order
\begin{subequations}\label{eq:order_on_dual}
\begin{eqnarray}
    \phi \geq 0 &&:\iff \phi(x)\geq 0, \quad \forall x \in \U, \, x \geq 0,\\
    \phi \geq \psi &&:\iff \phi - \psi \geq 0.
\end{eqnarray}
\end{subequations}
Furthermore, order intervals in $\U^*$ are weakly-* closed.
\end{lemma}
\begin{proof}
We need to check that $\phi \geq \psi \geq 0$ implies $\norm{\phi}_{\U^*} \geq \norm{\psi}_{\U^*}$. 
Splitting $x \in \U$ into positive and negative part as $x = x_+ - x_-$ with $x_\pm \geq 0$, we obtain by linearity and non-negativity that
$$\chi(x) = \chi(x_+) - \chi(x_-) \leq \chi(x_+), \quad \chi \in \{\phi,\psi\}.$$
This implies
$$\norm{\chi}_{\U^*} = \sup_{\substack{\norm{x}_\U=1}}\chi(x)=\sup_{\substack{\norm{x}_\U=1 \\ x \geq 0}}\chi(x), \quad \chi \in \{\phi, \psi\}.$$
Hence, we obtain
$$\norm{\phi}_{\U^*} = \sup_{\substack{\norm{x}_\U=1 \\ x \geq 0}}\phi(x) \geq \sup_{\substack{\norm{x}_\U=1 \\ x \geq 0}}\psi(x) = \norm{\psi}_{\U^*},$$
which proves that $\U^*$ is a Banach lattice. 
Now we prove weak-* closedness of order intervals. 
Here it is sufficient to show that whenever $(\phi_k) \subset \U^*$ converges weakly$^*$ to some $\phi\in\U^*$ and meets $\phi_k \geq 0$ for all $k \in \N$ it holds $\phi \geq 0$.
Using the assumptions we get
$$0 \leq \lim_{k\to\infty}\phi_k(x) = \phi(x),\quad \forall x \in \U,\,x\geq 0,$$
which according to \eqref{eq:order_on_dual} means $\phi \geq 0$.
\end{proof}


We also need the following result unrelated to Banach lattices.

\begin{lemma}\label{lem:weak-star_cont}
Let $A:\U^*\to\V^*$ be a bounded linear operator mapping between the duals of two Banach spaces $\U$ and $\V$, and let $J_\U$ and $J_\V$ be the canonical embeddings of $\U$ and $\V$ into $\U^{**}$ and $\V^{**}$.
If $A^*J_\V(\V) \subset J_\U(\U)$, then $A$ is weak-* to weak-* continuous.
\end{lemma}
\begin{proof}
Let $(\eta_k) \subset \U^*$ converge weakly-* to $\eta \in \U^*$. 
Using that for any $y \in \V$ it holds $A^*J_\V(y) = J_\U(x)$ for some $x\in\U$, we obtain
\begin{eqnarray*}
&&\langle A \eta, y \rangle_{\V^*,\V}=\langle J_\V(y), A \eta \rangle_{\V^{**},\V^*}=\langle A^* J_\V(y), \eta \rangle_{\U^{**},\U^*} = \langle J_\U(x), \eta \rangle_{\U^{**},\U^*} \\ 
&=&\langle \eta, x \rangle_{\U^*,\U} = \lim_{k\to\infty} \langle \eta_k, x\rangle_{\U^*,\U} = \lim_{k\to\infty} \langle J_\U(x), \eta_k\rangle_{\U^{**},\U^*} \\ 
&=&\lim_{k\to\infty} \langle A^* J_\V(y), \eta_k\rangle_{\U^{**},\U^*}=\lim_{k\to\infty} \langle J_\V(y), A\eta_k\rangle_{\V^{**},\V^*}=\lim_{k\to\infty} \langle A\eta_k, y\rangle_{\V^{*},\V},
\end{eqnarray*}
which means that $(A \eta_k)$ converges weakly-* to $A \eta$.
\end{proof}
\begin{remark}
A sufficient condition for $A^*J_\V(\V) \subset J_\U(\U)$ in Lemma~\ref{lem:weak-star_cont} is that $A=B^*$ for a bounded linear operator $B : \V \to \U$.
In this case $A^*=B^{**}:\V^{**}\to \U^{**}$ and it is easy to see that $B^{**}J_\V(y)=J_\U(By)$ for every $y\in\V$ which means $A^*J_\V(\V)=B^{**}J_\V(\V)=J_\U(B\V)\subset J_\U(\U)$.
\end{remark}

\end{document}